\PassOptionsToPackage{dvipsnames}{xcolor}
\pdfoutput=1
\documentclass{article}

\usepackage[utf8]{inputenc}
\usepackage{amssymb}
\usepackage{amsfonts}
\usepackage[tbtags]{amsmath}
\usepackage{amsthm}
\usepackage{romanbar}
\usepackage{mathrsfs}
\usepackage{tikz}
\usepackage{float}
\usepackage{colortbl}
\usepackage[left=3.5cm,right=3.5cm,top=3.5cm,bottom=3.5cm]{geometry}
\usetikzlibrary{cd}
\usetikzlibrary{decorations.pathmorphing}
\usetikzlibrary{positioning}
\usetikzlibrary{patterns}
\usetikzlibrary{patterns.meta}
\usepackage[backend=biber,style=alphabetic]{biblatex}
\usepackage{titlesec}
\usepackage[dvipsnames]{xcolor}
\usepackage{stmaryrd}
\usepackage{enumitem}
\usepackage{titletoc}
\usepackage{hyperref}

\newcommand\Hom[3][]{\mathrm{Hom}_{#1}(#2,#3)}
\newcommand\m[1]{\mathrm{mod}(#1)}

\newcommand\sm[1]{\underline{\mathrm{mod}}(#1)}
\newcommand\sgr[2]{\underline{\mathrm{mod}}^{\mathrm{#2}}(#1)}
\newcommand\proj[2]{\mathrm{proj}^{\mathrm{#2}}(#1)}
\newcommand\gr[2]{\mathrm{mod}^{\mathrm{#2}}(#1)}
\newcommand\Db[1]{\mathrm{D}^{\mathrm{b}}(#1)}
\newcommand\per[1]{\mathrm{per}(#1)}

\newcommand\id[1]{\mathrm{id}_{#1}}
\newcommand\coker[1]{\mathrm{coker}(#1)}
\newcommand\T[1]{\mathrm{Triv}(#1)}

\let\ens\mathbb
\newcommand\Z{\ens{Z}}

\tikzstyle{etiquette}=[minimum height=0.8cm, minimum width=1cm]

\newtheoremstyle{definition}
{\topsep}
{\topsep}
{\normalfont}
{0pt}
{\bfseries\fontfamily{bch}\selectfont}
{.}
{ }
{\thmname{#1}\thmnumber{ #2}\textnormal{\thmnote{ (#3)}}}

\newtheoremstyle{notation}{\topsep}{\topsep}{\normalfont}{0pt}{\bfseries\fontfamily{bch}\selectfont}{.}{ }{\thmname{#1}\thmnumber{ #2}\textnormal{\thmnote{ (#3)}}}

\newtheoremstyle{lemma}{\topsep}{\topsep}{\itshape}{0pt}{\bfseries\fontfamily{bch}\selectfont}{.}{ }{\thmname{#1}\thmnumber{ #2}\textnormal{\thmnote{ (#3)}}}

\newtheoremstyle{proposition}{\topsep}{\topsep}{\itshape}{0pt}{\bfseries\fontfamily{bch}\selectfont}{.}{ }{\thmname{#1}\thmnumber{ #2}\textnormal{\thmnote{ (#3)}}}

\newtheoremstyle{corollary}{\topsep}{\topsep}{\itshape}{0pt}{\bfseries\fontfamily{bch}\selectfont}{.}{ }{\thmname{#1}\thmnumber{ #2}\textnormal{\thmnote{ (#3)}}}

\newtheoremstyle{theorem}{\topsep}{\topsep}{\itshape}{0pt}{\bfseries\fontfamily{bch}\selectfont}{.}{ }{\thmname{#1}\thmnumber{ #2}\textnormal{\thmnote{ (#3)}}}

\newtheoremstyle{example}{\topsep}{\topsep}{\normalfont}{0pt}{\bfseries\fontfamily{bch}\selectfont}{.}{ }{\thmname{#1}\thmnumber{ #2}\textnormal{\thmnote{ (#3)}}}

\newtheoremstyle{remark}{\topsep}{\topsep}{\normalfont}{0pt}{\bfseries\fontfamily{bch}\selectfont}{.}{ }{\thmname{#1}\thmnumber{ #2}\textnormal{\thmnote{ (#3)}}}

\theoremstyle{definition}
\newtheorem{defn}{Definition}[section]

\theoremstyle{notation}

\theoremstyle{lemma}
\newtheorem{lem}[defn]{Lemma}

\theoremstyle{proposition}
\newtheorem{prop}[defn]{Proposition}
\newtheorem*{prop*}{Proposition}

\theoremstyle{corollary}
\newtheorem{cor}[defn]{Corollary}

\theoremstyle{theorem}
\newtheorem{thm}[defn]{Theorem}
\newtheorem*{thm*}{Theorem}

\theoremstyle{example}
\newtheorem{ex}[defn]{Example}

\theoremstyle{remark}
\newtheorem{rem}[defn]{Remark}

\titleformat{\section}[block]
{\centering\normalfont\Large\fontfamily{qhv}\selectfont}
{\thesection.}{0.5em}{}

\titleformat{\subsection}[block]
{\flushleft\normalfont\large\itshape\fontfamily{cmss}\selectfont}
{\thesubsection.}{0.5em}{}

\titlecontents{section}[0pt]{\bfseries\addvspace{0.1cm}}
{\thecontentslabel. \hspace{0.1mm}}
{}
{\hfill\contentspage}

\titlecontents{subsection}[15pt]{}
{\thecontentslabel. \hspace{0.1mm}}
{}
{\titlerule*[0.5pc]{.}\contentspage}

\title{\fontfamily{qhv}\selectfont Derived equivalence of algebras induced by their trivial extensions}
\author{\scshape Valentine Soto}
\date{}

\begin{document}
\maketitle

\medskip

\begin{center}
\begin{minipage}{0.75\textwidth}
\textbf{\fontfamily{bch}\selectfont Abstract.} The bounded derived category of a finite dimensional algebra of finite global dimension is equivalent the stable category of $\Z$-graded modules over its trivial extension \cite{Happel}. In particular, given two derived equivalent finite dimensional algebras $\Lambda_{1}$ and $\Lambda_{2}$ of finite global dimension, their trivial extensions are stable equivalent. The converse is not true in general. The goal of this paper is to study cases where derived equivalences between $\Lambda_{1}$ and $\Lambda_{2}$ arise from an equivalence of categories involving their trivial extension. Thanks to a graded version of Happel's theorem, we show that one can construct a $\Z$-grading on the $\Lambda_{i}$ so that such an equivalence involving their trivial extension yield a derived equivalence between the category of $\Z$-graded module over $\Lambda_{i}$. We describe explicitly the tilting object associated to this derived equivalence (in the non-graded and in the graded case) for triangular matrix algebras. Finally, we apply these results to the particular case of gentle algebras. In this context, we study how one can obtain a derived equivalence between $\Lambda_{1}$ and $\Lambda_{2}$ (in the non-graded and the graded case) from graded generalized Kauer moves.
\end{minipage}
\end{center}

\medskip

\tableofcontents

\phantomsection
\section*{Introduction}
\addcontentsline{toc}{section}{Introduction}

\medskip

Derived categories play a crucial role in different fields of mathematics like in algebra or geometry. They turn out to be important in the representation theory of algebras since they encode the homological properties of the algebra. Hence a natural but difficult question arises : determining when two algebras are derived equivalent. This problem was handled by generalizing the classical Morita theory to derived categories where the role of finitely generated projective generators is played by tilting complexes \cite{Rickard}. This has led to the emergence of the tilting theory \cite{AHHK}. This question was even answered in a more general context like for DG categories \cite{Keller} or more recently for the category of graded algebras with different group grading \cite{ZH}. In the recent years, another natural but difficult question has drawn some interest : establishing a complete derived invariant for some classes of algebras closed under derived equivalence. This problem has been solved for gentle algebras \cite{APS}, for Brauer graph algebras \cite{OZ} and for homologically smooth and proper graded gentle algebras, seen as DG algebras with trivial differential \cite{JSW}.

In this paper, we are interested in a result due to Happel which interprets the bounded derived category of a finite dimensional algebra of finite global dimension as the stable category of its repetitive algebra \cite{Happel}. One can understand modules over the repetitive algebra of $A$ as $\Z$-graded modules over the trivial extension of $A$ \cite{Keller2}. In particular, given two derived equivalent finite dimensional algebras of finite global dimension, their trivial extension are stable equivalent. However, we know that the converse is false in general. The goal of this paper is to study some cases where the converse holds. 

In the first section, we focus on cases where a derived equivalence between finite dimensional algebras of finite global dimension can be obtained from an equivalence of categories involving their trivial extension (Corollary \ref{cor:non gradué} and Corollary \ref{cor:non gradué projectif}). We extend these results for graded algebras. The key step in order to do so is to generalize Happel's theorem to the graded case.

\medskip

\begin{thm*}[Theorem \ref{thm:Happel gradué}]
    Let $G$ be an abelian group and $\Lambda$ be a basic $G$-graded finite dimensional algebra of finite global dimension. Then, there is an equivalence of triangulated categories

\smallskip

\[F:\Db{\gr{\Lambda}{G}}\overset{\sim}{\longrightarrow} \sgr{\T{\Lambda}}{\Z\times G}\]

\smallskip

\noindent such that $F\circ \tau[2]\overset{\sim}{\longrightarrow} (1,0)\circ F$ and $F\circ (g)\overset{\sim}{\longrightarrow} (0,g)\circ F$ for all $g\in G$, where $\tau$ denotes the Auslander-Reiten translation in $\Db{\gr{\Lambda}{G}}$.
\end{thm*} 

\medskip

Given two finite dimensional algebras $\Lambda_{1}$ and $\Lambda_{2}$ of finite global dimension with isomorphic trivial extension, one can construct under some assumptions a $\Z$-grading on $\Lambda_{1}$ from the natural $\Z$-grading of the trivial extension of $\Lambda_{2}$. Similarly, one can also construct a $\Z$-grading on $\Lambda_{2}$ from the natural $\Z$-grading of the trivial extension of $\Lambda_{1}$. Hence, the graded version of Happel's theorem leads to the following result.

\medskip

\begin{thm*}[Corollary \ref{cor:gradué iso}]
    Let $\Lambda_{1}$ and $\Lambda_{2}$ be two finite dimensional algebra of finite global dimension. In the above setting, there is an equivalence of triangulated categories

    \[\Db{\gr{\Lambda_{1}}{\Z}}\simeq \Db{\gr{\Lambda_{2}}{\Z}}\]
\end{thm*}

\medskip

These results in the non-graded and in the graded case have some similarities with the work of Amiot and Oppermann where they are interested in an analogous question in the context of cluster equivalent algebras of global dimension at most 2 \cite[Theorem 5.6 and Theorem 8.7]{AO}. To conclude the first section, we study the case of triangular matrix algebras. We describe explicitly the tilting object corresponding to the equivalence arising from Happel's theorem in the non-graded and in the graded case (Proposition \ref{prop:tilting object for triangulated matrices} and Proposition \ref{prop:tilting object for graded triangular matrices}). In this context, such an equivalence can be obtained by requiring only one of the two algebras to be of finite global dimension. In the non-graded case, we find again the result of Theorem 4.9 of \cite{Ladkani} for basic triangular matrix algebras of finite global dimension. 

In the second section, we apply the results obtained previously to the particular case of gentle algebras. These algebras are interesting in our context since their trivial extensions are known to be Brauer graph algebras with multiplicity identically one \cite{Schroll}. These are finite dimensional algebras that are constructed thanks to a combinatorial data called a Brauer graph. This combinatorial description of the algebra has led to the construction of a combinatorial tool, called the generalized Kauer moves, that yields derived equivalences of Brauer graph algebras from the moves of edges in their corresponding Brauer graphs \cite{Kauer,Soto}. The graded version of this result (Theorem \ref{thm:graded generalized Kauer moves}) has led to the following result.

\medskip

\begin{thm*}
Let $\Lambda_{1}$ and $\Lambda_{2}$ be two gentle algebras of finite global dimension and $(B_{1},d_{1})$, $(B_{2},d_{2})$ be their corresponding Brauer graph algebra equipped with an admissible cut. 

\smallskip

\begin{enumerate}
    \item (Corollary \ref{cor:non graded generalized Kauer moves}) Assume that $(B_{2},d_{2})$ can be obtained from a $\Z$-graded generalized Kauer move of $(B_{1},d_{1})$ up to an equivalence of the $\Z$-gradings. Then, $\Lambda_{1}$ and $\Lambda_{2}$ are derived equivalent.
    \item (Corollary \ref{cor:graded generalized Kauer moves}) Assume that $B_{2}$ can be obtained from a generalized Kauer move of $B_{1}$. Using the $\Z$-graded generalized Kauer moves, one can construct a $\Z$-grading on the algebras $\Lambda_{1}$ and $\Lambda_{2}$ such that there is an equivalence of triangulated categories

    \[\Db{\gr{\Lambda_{1}}{\Z}}\simeq\Db{\gr{\Lambda_{2}}{\Z}}\]
\end{enumerate}
\end{thm*}

\medskip

\noindent \textbf{Acknowledgments.} This paper is part of a PhD thesis supervised by Claire Amiot and supported by the CNRS. Part of this work was done while the author was in Sherbrooke University. This mobility was co funded by the Mitacs Globalink Research Award and the IDEX. The author thanks her advisor for the valuable advice and comments. She also thanks Thomas Brüstle for helpful discussions and the Department of Mathematics of Sherbrooke University for the warm welcome.

\medskip

\phantomsection
\section*{Notations and conventions}
\addcontentsline{toc}{section}{Notations and conventions}

\medskip

In this paper, all algebras are supposed to be finite dimensional over a field $k$. For any algebra $\Lambda$, we denote by $\m{\Lambda}$ the category of finitely generated right $\Lambda$-modules and $D=\Hom[k]{-}{k}$ the standard duality on $\m{\Lambda}$. If $\Lambda$ is self-injective, the stable category of $\m{\Lambda}$ will be denoted $\sm{\Lambda}$. This category has a natural structure of triangulated category \cite{Happel}. Moreover, $\Db{\Lambda}$ stands for the bounded derived category associated to $\m{\Lambda}$ and $\per{\Lambda}$ for the full subcategory of $\Db{\Lambda}$ consisting of perfect complexes over $\Lambda$. Given $\Lambda$ a graded algebra over an abelian group $G$, $\gr{\Lambda}{G}$ denotes the category of finitely generated $G$-graded modules over $\Lambda$ with morphisms of degree zero. Again, if $\Lambda$ is a self-injective $G$-graded algebra, the stable category of $\gr{\Lambda}{G}$ will be denoted $\sgr{\Lambda}{G}$ and has a natural structure of triangulated category \cite{Happel}. Moreover, we denote by $M_{j}$ the $j$-th degree component of $M\in\gr{\Lambda}{G}$ and $\vert m\vert$ the degree of any homogeneous element $m\in M$. Furthermore, $(g)$ stands for the shift functor of degree $g\in G$ in $\mathrm{mod^{G}}(\Lambda)$ where $M(g)=M$ as a $\Lambda$-module and $M(g)_{h}=M_{g+h}$ for all $h\in G$. By an equivalence of $G$-graded categories, we mean an equivalence between two $G$-graded categories that commutes with the shift functors. Finally, arrows in a quiver will be composed from right to left : for arrows $\alpha$, $\beta$, we write $\beta\alpha$ for the path from the start of $\alpha$ to the target of $\beta$. Similarly,
morphisms will be composed from right to left.

\medskip

\section{Derived equivalences induced by trivial extensions}

\medskip

The goal of this section is to study cases where one can obtain a derived equivalence between algebras from an equivalence of categories involving their trivial extension. Let us first recall some facts on trivial extensions.

\medskip

\begin{defn} \label{def:trivial extension}
Let $\Lambda$ be a $k$-algebra. The \textit{trivial extension} $\T{\Lambda}$ of $\Lambda$ by the $\Lambda$-$\Lambda$-bimodule $D\Lambda$ is the algebra defined by $\T{\Lambda}=\Lambda\oplus D\Lambda$ as a $k$-vector space and whose multiplication is induced by the $\Lambda$-$\Lambda$-bimodule structure of $D\Lambda$ i.e.

\[(a,\phi) \ . \ (b,\psi)=(ab, a\psi+\phi b)\]

\noindent for all $a,b\in \Lambda$ and $\phi,\psi\in D\Lambda$.
\end{defn}

\medskip

The trivial extension $\T{\Lambda}$ of any algebra $\Lambda$ is symmetric (see for instance \cite{Schiffler}) and has a natural $\Z$-grading concentrated in degree 0 and 1 where $\T{\Lambda}_{0}=\Lambda$, $\T{\Lambda}_{1}=D\Lambda$.

\medskip

\subsection{Non graded case}

\medskip

The key result of this subsection is given by Happel's theorem which links the bounded derived category of an algebra with the stable category of its trivial extension.

\medskip

\begin{thm}[Happel {\cite[Section II-4.9]{Happel}}; Keller {\cite[Section 9.7]{Keller2}}] \label{thm:Happel}
Let $\Lambda$ be a basic finite dimensional $k$-algebra of finite global dimension. Then, there is a triangle equivalence

\smallskip

\[F:\Db{\Lambda}\overset{\sim}{\longrightarrow} \sgr{\T{\Lambda}}{\Z}\]

\smallskip

\noindent such that $F\circ \tau[2]\overset{\sim}{\longrightarrow} (1)\circ F$, where $\tau$ denotes the Auslander-Reiten translation in $\Db{\Lambda}$.
\end{thm}

\medskip

In particular, given two basic finite dimensional $k$-algebras $\Lambda_{1}$ and $\Lambda_{2}$ of finite global dimension such that the stable categories $\sgr{\T{\Lambda_{i}}}{\Z}$ are equivalent as $\Z$-graded triangulated categories, we obtain the following commutative diagram

\smallskip

\begin{center}
    \begin{tikzcd}[row sep=1.5cm, column sep=1.5cm]
    \Db{\Lambda_{1}} \arrow{d}[left]{\wr} \arrow[dashed]{r}[above]{\sim} &\Db{\Lambda_{2}} \arrow{d}[right]{\wr} \\
    \sgr{\T{\Lambda_{1}}}{\Z} \arrow{r}[above]{\sim}&\sgr{\T{\Lambda_{2}}}{\Z}
    \end{tikzcd}
\end{center}

\smallskip

\noindent where the columns are the equivalence from Theorem \ref{thm:Happel}. One way to obtain such an equivalence between the stable categories $\sgr{\T{\Lambda_{i}}}{\Z}$ is to consider an equivalence of $\Z$-graded triangulated categories between the bounded derived categories of $\gr{\T{\Lambda_{i}}}{\Z}$ as follows.

\medskip

\begin{cor} \label{cor:non gradué}
Let $\Lambda_{1}$ and $\Lambda_{2}$ be two basic finite dimensional $k$-algebras of finite global dimension. If there is an equivalence of $\Z$-graded triangulated categories 

\[\Db{\gr{\T{\Lambda_{1}}}{\Z}}\simeq \Db{\gr{\T{\Lambda_{2}}}{\Z}}\]

\noindent then $\Lambda_{1}$ and $\Lambda_{2}$ are derived equivalent.
\end{cor}

\medskip

\begin{proof}
This follows from the equivalence of triangulated categories between the stable category $\sgr{A}{\Z}$ and the quotient category $\Db{\gr{A}{\Z}}/\per{\gr{A}{\Z}}$ for any $\Z$-graded self-injective algebra $A$ \cite{Rickard}. 
\end{proof}

\medskip

We now apply the previous result for an equivalence between the $\Z$-graded bounded derived category arising from a tilting object of the form $\oplus_{n}T(n)$ where $T=\oplus_{i} \ e_{i}\T{\Lambda}(n_{i})$.

\medskip

\begin{cor} \label{cor:non gradué projectif}
Let $\Lambda_{1}$ and $\Lambda_{2}$ be two basic finite dimensional $k$-algebras of finite global dimension such that there exists $\phi:\T{\Lambda_{1}}\rightarrow \T{\Lambda_{2}}$ an isomorphism of algebras. We denote by $\proj{\T{\Lambda_{i}}}{\Z}$ the full subcategory of $\gr{\T{\Lambda_{i}}}{\Z}$ whose objects are the projectives. Let us assume that $\phi$ sends an homogeneous element of $\T{\Lambda_{1}}$ onto an homogeneous element of $\T{\Lambda_{2}}$ and that for all $j$ there exists $n_{j}\in\Z$ such that 

\smallskip

\begin{equation*}
    \begin{aligned}
    \proj{\T{\Lambda_{1}}}{\Z} &\overset{\sim}{\longrightarrow} \proj{\T{\Lambda_{2}}}{\Z} \\[0.1cm]
    e_{j}\T{\Lambda_{1}} &\longmapsto \left(\phi(e_{j})\T{\Lambda_{2}}\right)(n_{j})
    \end{aligned}
\end{equation*}

\smallskip

\noindent gives an equivalence of $\Z$-graded categories. Then, the algebras $\Lambda_{1}$ and $\Lambda_{2}$ are derived equivalent.
\end{cor}

\medskip

\begin{ex} \label{ex:cor non gradué projectif}
    Let us consider $\Lambda_{1}$ and $\Lambda_{2}$ two finite dimensional $k$-algebras of finite global dimension defined by quiver and relations where the quivers are given by

    \smallskip

    \begin{center}
        $Q_{\Lambda_{1}}:$
        \begin{tikzcd}[column sep=1cm]
            1 &2 \arrow[cramped]{l}[above]{\alpha_{1}} \arrow[cramped]{r}[above]{\alpha_{2}} &3
        \end{tikzcd} 
        \qquad \mbox{and} \qquad 
        $Q_{\Lambda_{2}}:$
        \begin{tikzcd}[column sep=1cm]
            1 \arrow[cramped]{r}[above]{\beta_{1}} \arrow[phantom]{r}[name=Z]{}  &2 \arrow[cramped]{r}[above]{\beta_{2}} \arrow[phantom]{r}[name=A]{} \arrow[dash, densely dashed, bend right=45, cramped, from=Z]{A} &3
        \end{tikzcd} 
    \end{center}

    \smallskip

    \noindent with no relations for $\Lambda_{1}$ and with relation $\beta_{2}\beta_{1}=0$ for $\Lambda_{2}$. Note that these two algebras have the same trivial extension. Considering the natural $\Z$-grading of the trivial extension, we obtain the two following $\Z$-graded quivers for the trivial extension

    \smallskip

    \begin{center}
        $Q_{\T{\Lambda_{1}}} :$
        \begin{tikzcd}[column sep=1cm]
            1 \arrow[cramped, yshift=1mm]{r}[above]{1} &2 \arrow[cramped, yshift=-1mm]{l}[below]{0} \arrow[cramped, yshift=1mm]{r}[above]{0} &3 \arrow[cramped, yshift=-1mm]{l}[below]{1}
        \end{tikzcd}
        \qquad \mbox{and} \qquad
        $Q_{\T{\Lambda_{2}}} :$
        \begin{tikzcd}[column sep=1cm]
            1 \arrow[cramped, yshift=1mm]{r}[above]{0} &2 \arrow[cramped, yshift=-1mm]{l}[below]{1} \arrow[cramped, yshift=1mm]{r}[above]{0} &3 \arrow[cramped, yshift=-1mm]{l}[below]{1}
        \end{tikzcd}
    \end{center}

    \smallskip

    \noindent Thus, there is an equivalence of $\Z$-graded categories given by

    \smallskip
    
   \begin{equation*}
    \begin{aligned}
    \proj{\T{\Lambda_{1}}}{\Z} &\overset{\sim}{\longrightarrow} \proj{\T{\Lambda_{2}}}{\Z} \\[0.1cm]
    e_{1}\T{\Lambda_{1}} &\longmapsto e_{1}\T{\Lambda_{2}} && \\
    e_{j}\T{\Lambda_{1}} &\longmapsto (e_{j}\T{\Lambda_{2}})(1) &&\mbox{for $j=2,3$}
    \end{aligned}
\end{equation*}

    \smallskip

    \noindent In particular, these two algebras satisfy the condition of Corollary \ref{cor:non gradué projectif} for $\phi:\T{\Lambda_{1}}\rightarrow\T{\Lambda_{2}}$ being the identity. Hence, $\Lambda_{1}$ and $\Lambda_{2}$ are derived equivalent.
\end{ex}

\medskip

\begin{rem} \label{rem:Happel Rickard}
    Let $\Lambda_{1}$ and $\Lambda_{2}$ be two algebras satisfying the assumptions of Corollary \ref{cor:non gradué projectif}. Thus, we obtain the following commutative diagram \cite[Section 9.7]{Keller2}

    \smallskip

\begin{center}
    \begin{tikzcd}[row sep=1.5cm, column sep=2cm]
    \Db{\Lambda_{1}} \arrow{d}[left]{\wr} \arrow[dashed]{r}[above]{\sim} &\Db{\Lambda_{2}} \arrow{d}[right]{\wr} \\
    \sgr{\T{\Lambda_{1}}}{\Z} \arrow{d} \arrow{r}[above]{\sim}&\sgr{\T{\Lambda_{2}}}{\Z} \arrow{d} \\
    \sm{\T{\Lambda_{1}}} \arrow{r}[above]{\sim} &\sm{\T{\Lambda_{2}}}
    \end{tikzcd}
\end{center}

\smallskip

\noindent where the top commutative square is given by the equivalence of Theorem \ref{thm:Happel}. In particular, the derived equivalence between the $\Lambda_{i}$ yields the existence of a tilting object $T\in\per{\Lambda_{2}}$ whose endomorphism ring is isomorphic to $\Lambda_{1}$. Thanks to Theorem 3.1 of \cite{Rickard}, we also obtain the following commutative diagram

\smallskip

\begin{center}
    \begin{tikzcd}[row sep=2.3cm, column sep=3cm]
    \per{\Lambda_{1}} \arrow{r}[above]{-\underset{\Lambda_{1}}{\overset{\mbox{\tiny \bf{L}}}{\otimes}}T} \arrow{d}[left, xshift=-1mm]{-\underset{\Lambda_{1}}{\overset{\mbox{\tiny \bf{L}}}{\otimes}}\T{\Lambda_{1}}} &\per{\Lambda_{2}} \arrow{d}[right, xshift=1mm]{-\underset{\Lambda_{2}}{\overset{\mbox{\tiny \bf{L}}}{\otimes}}\T{\Lambda_{2}}} \\
    \per{\T{\Lambda_{1}}} \arrow{r}[below]{-\underset{\T{\Lambda_{1}}}{\overset{\mbox{\tiny \bf{L}}}{\otimes}}(T\,\underset{\Lambda_{2}}{\overset{\mbox{\tiny \bf{L}}}{\otimes}}\,\T{\Lambda_{2}})} &\per{\T{\Lambda_{2}}}
    \end{tikzcd}
\end{center}

\smallskip

\noindent Using Theorem 2.1 of \cite{Rickard}, we deduce an equivalence of triangulated categories

\smallskip

\begin{center}
    \begin{tikzcd}[column sep=3cm]
        \sm{\T{\Lambda_{1}}} \arrow[cramped]{r}[above]{-\underset{\T{\Lambda_{1}}}{\overset{\mbox{\tiny \bf{L}}}{\otimes}}(T\,\underset{\Lambda_{2}}{\overset{\mbox{\tiny \bf{L}}}{\otimes}}\,\T{\Lambda_{2}})} &\sm{\T{\Lambda_{2}}}
    \end{tikzcd}
\end{center}

\smallskip

\noindent which may be different from the one induced by the first commutative square that arises from Happel's theorem.

\end{rem}

\medskip

\subsection{Graded case}

\medskip

In this subsection, we generalize the result of the previous subsection for graded algebras. Let us consider $\Lambda$ a graded finite dimensional $k$-algebra over an abelian group $G$. The $G$-grading of $\Lambda$ induces a $G$-grading on its dual $D\Lambda=\Hom[k]{\Lambda}{k}$ as follows

\[\left(D\Lambda\right)_{g}:=D(\Lambda_{-g})=\Hom[k]{\Lambda_{-g}}{k}\]

\noindent for all $g\in G$. This gives rise to a $G$-grading on the trivial extension $\T{\Lambda}$ of $\Lambda$ where

\[\T{\Lambda}_{g}:=\Lambda_{g}\oplus (D\Lambda)_{g}=\Lambda_{g}\oplus D(\Lambda_{-g})\]

\noindent for all $g\in G$. In particular, for $(b_{1},\ldots, b_{n})$ a homogeneous basis of $\Lambda$, $(b_{1}^{*},\ldots, b_{n}^{*})$ is a homogeneous basis of $D\Lambda$ where the degree $\vert b_{i}^{*}\vert$ of $b_{i}^{*}$ is $-\vert b_{i}\vert $. Thus, $(b_{1},\ldots, b_{n},b_{1}^{*},\ldots, b_{n}^{*})$ is a homogeneous basis of $\T{\Lambda}$ for the $G$-grading induced by $\Lambda$. Considering the natural $\Z$-grading of $\T{\Lambda}$, we have a $(\Z\times G)$-grading on the basis $(b_{1},\ldots, b_{n},b_{1}^{*},\ldots, b_{n}^{*})$ of $\T{\Lambda}$ which extends to a $(\Z\times G)$-grading on $\T{\Lambda}$. We denote by $(p,g)$ the shift functor of degree $(p,g)\in\Z\times G$ in $\gr{\T{\Lambda}}{\Z\times G}$.

\medskip

The key step is to generalize Happel's theorem to the graded case.

\medskip

\begin{thm}[Graded version of Happel theorem] \label{thm:Happel gradué}
Let $G$ be an abelian group and $\Lambda$ be a basic $G$-graded finite dimensional $k$-algebra of finite global dimension. Then, there is an equivalence of triangulated categories

\smallskip

\[F:\Db{\gr{\Lambda}{G}}\overset{\sim}{\longrightarrow} \sgr{\T{\Lambda}}{\Z\times G}\]

\smallskip

\noindent such that $F\circ \tau[2]\overset{\sim}{\longrightarrow} (1,0)\circ F$ and $F\circ (g)\overset{\sim}{\longrightarrow} (0,g)\circ F$ for all $g\in G$, where $\tau$ denotes the Auslander-Reiten translation in $\Db{\gr{\Lambda}{G}}$.
\end{thm}

\medskip

In \cite{Happel}, the equivalence of Theorem \ref{thm:Happel} is constructed using a category equivalent to $\gr{\T{\Lambda}}{\Z}$. We begin with adapting this equivalence of categories to the graded case. Let $\Lambda$ be a graded finite dimensional $k$-algebra over an abelian group $G$. For all $M,N\in\gr{\Lambda}{G}$, $\Hom[\Lambda]{M}{N}$ is a finitely generated $\Lambda$-module. Moreover, since $M$ is finitely generated and $M$ and $N$ are $G$-graded, $\Hom[\Lambda]{M}{N}$ has a structure of $G$-graded $\Lambda$-module. Let us 
define the category $\mathscr{C}^{G}(\Lambda)$ as follows

\smallskip

\begin{itemize}[label=\textbullet, font=\tiny]
\item The objects of $\mathscr{C}^{G}(\Lambda)$ are sequences $(M_{n},f_{n})_{n\in\Z}$ where the $M_{n}$ are $G$-graded $\Lambda$-modules, all but a finite number being zero and the $f_{n}:M_{n}\rightarrow\Hom[\Lambda]{D\Lambda}{M_{n+1}}$ are morphisms in $\gr{\Lambda}{G}$ satisfying $\Hom[\Lambda]{D\Lambda}{f_{n+1}}\circ f_{n}=0$ for all $n\in\Z$. They can be represented in the following way 

\smallskip

    \begin{center}
    \begin{tikzcd}
    \cdots \arrow[phantom]{r}[description]{\sim} &M_{-2} \arrow[phantom]{r}[description]{\sim}[yshift=3mm, font=\scriptsize]{f_{-2}} &M_{-1} \arrow[phantom]{r}[description]{\sim}[yshift=3mm, font=\scriptsize]{f_{-1}} &M_{0} \arrow[phantom]{r}[description]{\sim}[yshift=3mm, font=\scriptsize]{f_{0}} &M_{1} \arrow[phantom]{r}[description]{\sim}[yshift=3mm, font=\scriptsize]{f_{1}} &M_{2} \arrow[phantom]{r}[description]{\sim}[yshift=3mm, font=\scriptsize]{f_{2}} &\cdots
    \end{tikzcd}
\end{center}

\smallskip

\item The morphisms of $\mathscr{C}^{G}(\Lambda)$ from $(M_{n},f_{n})_{n\in\Z}$ to $(N_{n},g_{n})_{n\in\Z}$ are sequences $(h_{n})_{n\in\Z}$ where the $h_{n}:M_{n}\rightarrow N_{n}$ are morphisms in $\gr{\Lambda}{G}$ such that the following diagram commutes for all $n\in\Z$.

\smallskip

\begin{center} 
    \begin{tikzcd}[row sep=1.7cm, column sep=2cm, font=\normalsize, every label/.append style = {font = \small}]
        M_{n} \arrow[r, "f_{n}"] \arrow[d, labels=left,  " h_{n}"]
        & \Hom[\Lambda]{D\Lambda}{M_{n+1}} \arrow[d, shift right= 0.7cm, "{\Hom[\Lambda]{D\Lambda}{h_{n+1}}}"] \\
        N_{n} \arrow[r, "g_{n}"] 
        & \Hom[\Lambda]{D\Lambda}{N_{n+1}}
    \end{tikzcd}
\end{center}
\end{itemize}

\medskip

\begin{lem} \label{lem:équivalence de catégories}
    The categories $\gr{\T{\Lambda}}{\Z\times G}$ and $\mathscr{C}^{G}(\Lambda)$ are equivalent.
\end{lem}

\medskip

To prove Theorem \ref{thm:Happel gradué}, we want to generalize Lemma 4.1 of \cite{Happel} which is the key step to construct the equivalence of triangulated categories of Theorem \ref{thm:Happel}. We will need the following result on projectives and injectives in the category $\mathscr{C}^{G}(\Lambda)$ which is an analogous version of Lemma 2.2 in \cite{Happel}.

\medskip

\begin{lem} \label{lem:projectif et injectif de la nouvelle catégorie}
    The category $\mathscr{C}^{G}(\Lambda)$ is a Frobenius category whose indecomposable projective-injective objects are of the form 

\smallskip

\begin{center}
    \begin{tikzcd}
    \cdots \arrow[phantom]{r}[description]{\sim} &0 \arrow[phantom]{r}[description]{\sim} &M_{i} \arrow[phantom]{r}[description]{\sim}[yshift=3mm, font=\scriptsize]{\id{M_{i}}} &M_{i+1} \arrow[phantom]{r}[description]{\sim} &0 \arrow[phantom]{r}[description]{\sim} &\cdots
    \end{tikzcd}
\end{center}

\smallskip

\noindent for some $i\in\Z$, where $M_{i+1}$ is indecomposable and injective in $\gr{\Lambda}{G}$ and $M_{i}$ is equal to $\Hom[\Lambda]{D\Lambda}{M_{i+1}}$.
\end{lem}

\medskip

\begin{lem} \label{lem:construction d'un foncteur injectif}
   There exists an exact functor $I:\mathscr{C}^{G}(\Lambda)\rightarrow\mathscr{C}^{G}(\Lambda)$ such that $I(X)$ is injective in $\mathscr{C}^{G}(\Lambda)$ for all $X\in\mathscr{C}^{G}(\Lambda)$ and a monomorphism $\mu:\id{\mathscr{C}^{G}(\Lambda)}\rightarrow I$.
\end{lem}

\medskip

\begin{proof}
    Let us consider the functor $I:\mathscr{C}^{G}(\Lambda)\rightarrow\mathscr{C}^{G}(\Lambda)$ defined as follows

    \smallskip
    
    \begin{itemize}[label=\textbullet, font=\tiny]
    \item For each $X=(X_{n},f_{n})_{n\in\Z}$ in $\mathscr{C}^{G}(\Lambda)$, $I(X):=(I(X)_{n},d_{n})_{n\in\Z}$ where

    \smallskip
    
    \begin{equation*}
        I(X)_{n}=\Hom[k]{D\Lambda}{X_{n+1}}\oplus\Hom[k]{\Lambda}{X_{n}}
        \qquad \mbox{and} \qquad d_{n}=\begin{pmatrix}
            0 &0 \\[0.5em]
            \delta_{n} & 0
        \end{pmatrix}
    \end{equation*}

    \smallskip

    \noindent for all $n\in\Z$ with $\delta_{n}:\Hom[k]{D\Lambda}{X_{n+1}}\rightarrow \Hom[\Lambda]{D\Lambda}{\Hom[k]{\Lambda}{X_{n+1}}}$ the canonical isomorphism of $\Lambda$-modules defined for all $\phi\in \Hom[k]{D\Lambda}{X_{n+1}}$ by 

    \begin{equation*}
        \delta_{n}(\phi)=(\psi\mapsto (a\mapsto \phi(\psi \ . \ a)))
    \end{equation*}

    \item For each morphism $h=(h_{n})_{n}$ in $\mathscr{C}^{G}(\Lambda)$ from $X=(X_{n},f_{n})_{n\in\Z}$ to $Y=(Y_{n},g_{n})_{n\in\Z}$, $I(h):=(I(h)_{n})_{n\in\Z}$ where the $I(h)_{n}:I(X)_{n}\rightarrow I(Y)_{n}$ are defined by

    \smallskip
    
    \begin{equation*}
        I(h)_{n}=\begin{pmatrix}
            \Hom[k]{D\Lambda}{h_{n+1}} & 0 \\[1em] 0 & \Hom[k]{\Lambda}{h_{n}}
        \end{pmatrix}
    \end{equation*}

    \end{itemize}

    \smallskip

    \noindent It is straightforward that $I$ is a well-defined exact functor using the definition of $\mathscr{C}^{G}(\Lambda)$. Since $\Lambda$ is finite dimensional and the $\Hom[k]{\Lambda}{X_{n}}$ are $G$-graded $\Lambda$-modules that are injective as $\Lambda$-modules, they are injective in $\gr{\Lambda}{G}$. Using the description of indecomposable projective-injective in $\mathscr{C}^{G}(\Lambda)$ given in Lemma \ref{lem:projectif et injectif de la nouvelle catégorie}, we conclude that $I(X)$ is indeed injective in $\mathscr{C}^{G}(\Lambda)$ for all $X\in\mathscr{C}^{G}(\Lambda)$. 

    It remains to construct a monomorphism $\mu:\id{\mathscr{C}^{G}(\Lambda)}\rightarrow I$. For all $X=(X_{n},f_{n})\in\mathscr{C}^{G}(\Lambda)$, let us define $\mu(X)=(\mu(X)_{n})_{n\in\Z}$ where the $\mu(X)_{n}:X_{n}\rightarrow I(X)_{n}$ are given by the matrices $(f_{n} \quad \xi_{n})$ with $\xi_{n}:X_{n}\rightarrow \Hom[k]{\Lambda}{X_{n}}$ defined by $\xi_{n}(x)=(a\mapsto a \ . \ x)$ for all $x\in X_{n}$. One can easily check that $\mu$ is a well-defined monomorphism.
    \end{proof}

    \medskip

    With the notations of the previous lemma, we define the functor $S=\coker{\mu}:\mathscr{C}^{G}(\Lambda)\rightarrow \mathscr{C}^{G}(\Lambda)$. This is an exact functor that sends injectives onto injectives. Thus, it induces a functor $\underline{S}:\underline{\mathscr{C}^{G}(\Lambda)}\rightarrow \underline{\mathscr{C}^{G}(\Lambda)}$ between the stable category $\underline{\mathscr{C}^{G}(\Lambda)}$ of $\mathscr{C}^{G}(\Lambda)$, which is the suspension functor of the triangulated category $\underline{\mathscr{C}^{G}(\Lambda)}$ used to construct the equivalence of Theorem \ref{thm:Happel gradué}. 
    
    \medskip

    \begin{proof}[Proof of Theorem \ref{thm:Happel gradué}]
    The construction of the equivalence follows the same steps as in the non-graded case \cite{Happel}. For more details, we refer to \cite{thèseSoto}.
    \end{proof}

\medskip

By construction, the equivalence of Theorem \ref{thm:Happel gradué} sends any object of $\gr{\Lambda}{G}$ onto itself. Similarly to the non-graded case, Theorem \ref{thm:Happel gradué} induces the following commutative diagram for any basic $G$-graded finite dimensional $k$-algebras $\Lambda_{1}$ and $\Lambda_{2}$ of finite global dimension such that the 
$\sgr{\T{\Lambda_{i}}}{\Z\times G}$ are equivalent as $\Z$-graded triangulated categories

\smallskip

\begin{center}
    \begin{tikzcd}[row sep=1.5cm, column sep=1.5cm]
    \Db{\gr{\Lambda_{1}}{G}} \arrow{d}[left]{\wr} \arrow[dashed]{r}[above]{\sim} &\Db{\gr{\Lambda_{2}}{G}} \arrow{d}[right]{\wr} \\
    \sgr{\T{\Lambda_{1}}}{\Z\times G} \arrow{r}[above]{\sim}&\sgr{\T{\Lambda_{2}}}{\Z\times G}
    \end{tikzcd}
\end{center}

\smallskip

\noindent where the columns are the equivalence from Theorem \ref{thm:Happel gradué}. This leads to the following result which is the graded version of Corollary \ref{cor:non gradué}.

\medskip

\begin{cor} \label{cor:gradué}
Let $\Lambda_{1}$ and $\Lambda_{2}$ be two basic $G$-graded finite dimensional $k$-algebras of finite global dimension. If there is an equivalence of $\Z$-graded triangulated categories

\[\Db{\gr{\T{\Lambda_{1}}}{\Z\times G}}\simeq\Db{\gr{\T{\Lambda_{2}}}{\Z\times G}}\]

\noindent then there is a triangle equivalence

\smallskip

\[\Db{\gr{\Lambda_{1}}{G}}\simeq \Db{\gr{\Lambda_{2}}{G}}\]

\end{cor}

\medskip

In the following, we construct a $\Z$-grading on the algebras such that the previous result can be applied for a $\Z^{2}$-graded derived equivalence arising from a tilting object of the form $\oplus_{x,y\in\Z} \ \T{\Lambda}(\theta(x,y))$. 

\medskip

Let $\Lambda_{1}$ and $\Lambda_{2}$ be two finite dimensional $k$-algebras such that there exists $\phi:\T{\Lambda_{1}}\rightarrow \T{\Lambda_{2}}$ an isomorphism of algebras. Assume that there exists $B_{1}=B_{1}^{0}\cup B_{1}^{1}$ basis of $\Lambda_{1}$ such that $B_{2}:=\phi(B_{1}^{0}\cup (B_{1}^{1})^{*}_{B_{1}})$ is a basis of $\Lambda_{2}$ and 

\begin{equation*} \label{eq:grading algebras}
\forall x\in B_{1}^{0}\, , \,\phi(x^{*}_{B_{1}})=\phi(x)^{*}_{B_{2}} \qquad \mbox{and} \qquad
\forall x\in B_{1}^{1}\, , \, \phi(x)=\phi(x_{B_{1}}^{*})^{*}_{B_{2}}
\end{equation*}

\noindent By assumption we know that $\phi(B_{1}^{0})\subset \Lambda_{2}$ and $\phi(B_{1}^{1})\subset D\Lambda_{2}$. Thus, the natural $\Z$-grading of $\T{\Lambda_{2}}$ induces a $\Z$-grading $d_{1}$ on the basis $B_{1}$ of $\Lambda_{1}$ as follows

\smallskip

\begin{equation*}
d_{1}(x)=\left\{ \begin{aligned} &0 &&\mbox{if $x\in B_{1}^{0}$} \\
&1 &&\mbox{if $x\in B_{1}^{1}$}
\end{aligned}\right.
\end{equation*}

\smallskip

\noindent This can be extended to a $\Z$-grading on $\Lambda_{1}$. We denote by $(\Lambda_{1},d_{1})$ the corresponding $\Z$-graded algebra. Hence, one can construct a $\Z^{2}$-grading on $\T{\Lambda_{1}}$ where the first component represents the natural $\Z$-grading of the trivial extension. Similarly, the natural $\Z$-grading of $\T{\Lambda_{1}}$ induces a $\Z$-grading $d_{2}$ on the basis $B_{2}=\phi(B_{1}^{0}\,\cup\, (B_{1}^{1})^{*}_{B_{1}})$ of $\Lambda_{2}$ as follows

\begin{equation*}
\forall x\in B_{1}^{0}\, , \, d_{2}(\phi(x))=0 \qquad \mbox{and} \qquad \forall x\in B_{1}^{1}\, , \, d_{2}(\phi(x_{B_{1}}^{*}))=1
\end{equation*}

\noindent Again, this can be extended to a $\Z$-grading on $\Lambda_{2}$. We denote by $(\Lambda_{2},d_{2})$ the corresponding $\Z$-graded algebra. Hence, one can again construct a $\Z^{2}$-grading on $\T{\Lambda_{2}}$ where the first component represents the natural $\Z$-grading of the trivial extension.

\medskip

\begin{cor} \label{cor:gradué iso}
Let $\Lambda_{1}$ and $\Lambda_{2}$ be two basic $k$-algebras of finite global dimension satisfying the previous hypotheses. Then, we have an equivalence of triangulated categories

\smallskip

\[F:\Db{\gr{\Lambda_{1}, d_{1}}{\Z}}\overset{\sim}{\longrightarrow} \Db{\gr{\Lambda_{2}, d_{2}}{\Z}}\]

\noindent which satisfy that $F\circ (1)\overset{\sim}{\longrightarrow} \tau_{2}[2](-1)\circ F$, where $\tau_{2}$ denotes the Auslander-Reiten translation in $\Db{\gr{\Lambda_{2}, d_{2}}{\Z}}$.
\end{cor}

\medskip

\begin{proof}
Note that the following isomorphism

\smallskip

\begin{equation*}
    \begin{aligned}
        &&\Z^{2} &&&\longrightarrow &&\Z^{2} \\
        &&(x,y) &&&\longmapsto && (x+y,-y)
    \end{aligned}
\end{equation*}

\smallskip

\noindent sends the $\Z^{2}$-grading of $\T{\Lambda_{1}}$ onto the $\Z^{2}$-grading of $\T{\Lambda_{2}}$. This induces an equivalence of triangulated categories

\[G:\sgr{\T{\Lambda_{1}}}{\Z^{2}}\overset{\sim}{\longrightarrow}\sgr{\T{\Lambda_{2}}}{\Z^{2}}\]

\noindent which satisfy that $G\circ (1,0)\overset{\sim}{\longrightarrow} (1,0)\circ G$ and $G\circ (0,1)\overset{\sim}{\longrightarrow} (1,-1)\circ G$. We conclude thanks to Theorem \ref{thm:Happel gradué}.
\end{proof}

\medskip

\begin{ex} \label{ex:cor gradué iso}
    Let us consider $\Lambda_{1}$ and $\Lambda_{2}$ two finite dimensional $k$-algebras of finite global dimension defined by quiver and relations as follows

    \smallskip

    \begin{center}
        $Q_{\Lambda_{1}} :$
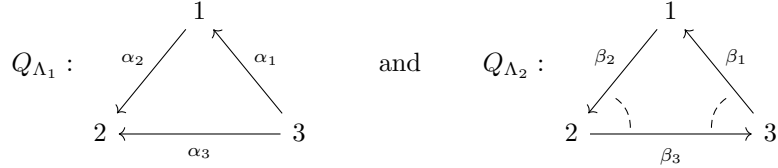
        \begin{tikzcd}
            &1 \arrow{dl}[above left]{\alpha_{2}} &  \\[0.5cm]
            2 & & 3 \arrow{ll}[below, midway, yshift=-1mm]{\alpha_{3}} \arrow{ul}[above right]{\alpha_{1}}
        \end{tikzcd}
        \qquad \mbox{and} \qquad 
        $Q_{\Lambda_{2}} :$
        \begin{tikzcd}
            &1 \arrow{dl}[above left]{\beta_{2}} \arrow[phantom]{dl}[near end, name=B]{} & \\[0.5cm]
            2 \arrow{rr}[below, midway, yshift=-1mm]{\beta_{3}} \arrow[phantom]{rr}[near start, name=A]{} \arrow[phantom]{rr}[near end, name=C]{} & & 3  \arrow{ul}[above right]{\beta_{1}} \arrow[phantom]{ul}[near start, name=D]{} \arrow[dash, densely dashed, cramped, bend right, from=A]{B} \arrow[dash, densely dashed, cramped, bend left, from=C]{D}
        \end{tikzcd}
    \end{center}
    
    \noindent with no relations for $\Lambda_{1}$ and with relations $\beta_{3}\beta_{2}=0$ and $\beta_{1}\beta_{3}=0$ for $\Lambda_{2}$. These two algebras have the same trivial extension. Considering the natural $\Z$-grading of the trivial extension, we obtain the following $\Z$-grading quivers for the trivial extension
    
     \smallskip

    \begin{center}
        $Q_{\T{\Lambda_{1}}} :$
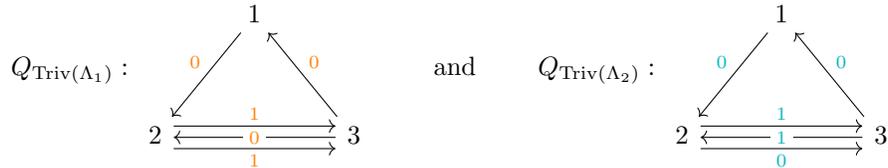
        \begin{tikzcd}
            &1 \arrow{dl}[above left, orange]{0} &  \\[0.5cm]
            2 & & 3 \arrow{ll}[orange, midway,description]{0} \arrow{ul}[above right, orange]{0}
            \arrow[from=2-1, yshift=1.5mm]{2-3}[above, midway, orange]{1} \arrow[from=2-1, yshift=-1.5mm]{2-3}[below, midway, orange]{1}
        \end{tikzcd}
        \qquad \mbox{and} \qquad 
         $Q_{\T{\Lambda_{2}}} :$
        \begin{tikzcd}
            &1 \arrow{dl}[above left, Aquamarine]{0} &  \\[0.5cm]
            2 & & 3 \arrow{ll}[Aquamarine, midway,description]{1} \arrow{ul}[above right, Aquamarine]{0}
            \arrow[from=2-1, yshift=1.5mm]{2-3}[above, midway, Aquamarine]{1} \arrow[from=2-1, yshift=-1.5mm]{2-3}[below, midway, Aquamarine]{0}
        \end{tikzcd}
    \end{center}
    
    \noindent Note that the algebras $\Lambda_{1}$ and $\Lambda_{2}$ satisfy the conditions of Corollary \ref{cor:gradué iso} for $\phi:\T{\Lambda_{1}}\rightarrow \T{\Lambda_{2}}$ being the identity and $B_{1}=B_{1}^{0} \cup B_{1}^{1}$ being the basis of paths in $\Lambda_{1}$ with $B_{1}^{0}=\{\varepsilon_{1},\varepsilon_{2},\varepsilon_{3}, \alpha_{1},\alpha_{2},\alpha_{2}\alpha_{1}\}$ and $B_{1}^{1}=\{\alpha_{3}\}$ where $\varepsilon_{j}$ denotes the trivial path at the vertex $j=1,2,3$. Hence, considering the indices modulo 2, the natural $\Z$-grading of $\T{\Lambda_{i+1}}$ induces a $\Z$-grading $d_{i}$ on the quiver of $\Lambda_{i}$ as follows

    \smallskip
    
    \begin{center}
        $Q_{\Lambda_{1}} :$
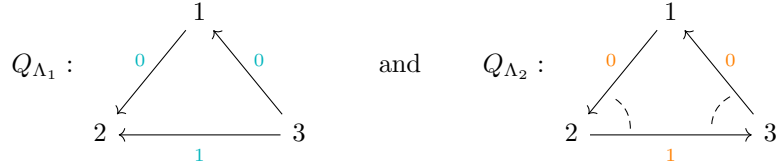
        \begin{tikzcd}
            &1 \arrow{dl}[above left, Aquamarine]{0} &  \\[0.5cm]
            2 & & 3 \arrow{ll}[below, midway, yshift=-1mm, Aquamarine]{1} \arrow{ul}[above right, Aquamarine]{0}
        \end{tikzcd}
        \qquad \mbox{and} \qquad 
        $Q_{\Lambda_{2}} :$
        \begin{tikzcd}
            &1 \arrow{dl}[above left, orange]{0} \arrow[phantom]{dl}[near end, name=B]{}& \\[0.5cm]
            2 \arrow{rr}[below, midway, yshift=-1mm, orange]{1} \arrow[phantom]{rr}[near start, name=A]{}[near end, name=D]{}& & 3  \arrow{ul}[above right, orange]{0} \arrow[phantom]{ul}[near start, name=C]{} \arrow[dash, densely dashed, cramped, bend right, from=A]{B} \arrow[dash, densely dashed, cramped, bend right, from=C]{D}
        \end{tikzcd}
    \end{center}

    \smallskip

    \noindent Since the relations of the $\Lambda_{i}$ are homogeneous for $d_{i}$, the $\Z$-grading can be extended to $\Lambda_{i}$. Thus, one can construct a $\Z^{2}$-grading on $\T{\Lambda_{i}}$ where the first component is the natural $\Z$-grading of the trivial extension. These $\Z^{2}$-gradings are defined on the quivers of $\T{\Lambda_{i}}$ as follows
    
    \smallskip

    \begin{center}
        $Q_{\T{\Lambda_{1}}} :$
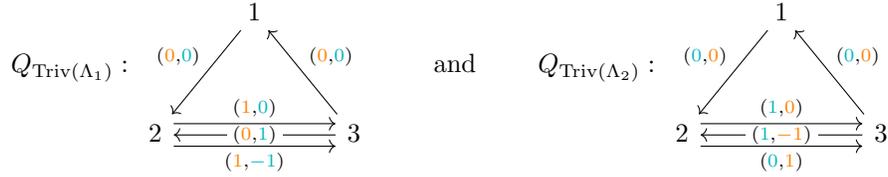
        \begin{tikzcd}
            &1 \arrow{dl}[above left]{(\textcolor{orange}{0},\textcolor{Aquamarine}{0})} &  \\[0.5cm]
            2 & & 3 \arrow{ll}[midway,description]{(\textcolor{orange}{0},\textcolor{Aquamarine}{1})} \arrow{ul}[above right]{(\textcolor{orange}{0},\textcolor{Aquamarine}{0})}
            \arrow[from=2-1, yshift=1.5mm]{2-3}[above, midway]{(\textcolor{orange}{1},\textcolor{Aquamarine}{0})} \arrow[from=2-1, yshift=-1.5mm]{2-3}[below, midway]{(\textcolor{orange}{1},\textcolor{Aquamarine}{-1})}
        \end{tikzcd}
        \qquad \mbox{and} \qquad 
         $Q_{\T{\Lambda_{2}}} :$
        \begin{tikzcd}
            &1 \arrow{dl}[above left]{(\textcolor{Aquamarine}{0},\textcolor{orange}{0})} &  \\[0.5cm]
            2 & & 3 \arrow{ll}[midway,description]{(\textcolor{Aquamarine}{1},\textcolor{orange}{-1})} \arrow{ul}[above right]{(\textcolor{Aquamarine}{0},\textcolor{orange}{0})}
            \arrow[from=2-1, yshift=1.5mm]{2-3}[above, midway]{(\textcolor{Aquamarine}{1},\textcolor{orange}{0})} \arrow[from=2-1, yshift=-1.5mm]{2-3}[below, midway]{(\textcolor{Aquamarine}{0},\textcolor{orange}{1})}
        \end{tikzcd}
    \end{center}

    \smallskip

    \noindent Note that the following isomorphism 

    \smallskip
    
    \begin{equation*}
    \begin{aligned}
        &&\Z^{2} &&&\longrightarrow &&\Z^{2} \\
        &&(x,y) &&&\longmapsto && (x+y,-y)
    \end{aligned}
\end{equation*}

\smallskip

\noindent sends the $\Z^{2}$-grading of $\T{\Lambda_{1}}$ onto the $\Z^{2}$-grading of $\T{\Lambda_{2}}$. Hence, we obtain the following equivalence of triangulated categories thanks to Theorem \ref{thm:Happel gradué}

\[\Db{\gr{\Lambda_{1},d_{1}}{\Z}}\simeq\Db{\gr{\Lambda_{2},d_{2}}{\Z}}\]

\noindent However, one can check that the algebras $\Lambda_{1}$ and $\Lambda_{2}$ are not derived equivalent. Indeed, since the quiver of $\Lambda_{2}$ has an oriented cycle, $\Lambda_{2}$ cannot be derived equivalent to a hereditary algebra. 
\end{ex}

\medskip

\subsection{Triangular matrix algebra}

\medskip

In this part, we apply the results from the two previous subsections to triangular matrix algebras. We are interested in these algebras since we can describe explicitely in this case the tilting object associated to the derived equivalence arising from Happel's theorem.

\medskip

\begin{prop} \label{prop:triangular matrices}
    Let $\Lambda_{1}$ and $\Lambda_{2}$ be two finite dimensional $k$-algebras of the form 

\smallskip

\begin{equation*}
    \Lambda_{1}=\begin{pmatrix}
        e^{+}\Lambda_{1}e^{+} &0 \\[0.9em]
        e^{-}\Lambda_{1}e^{+} &e^{-}\Lambda_{1}e^{-}
    \end{pmatrix} \qquad \mbox{and} \qquad
    \Lambda_{2}=\begin{pmatrix}
        e^{+}\Lambda_{1}e^{+} &e^{+}D\Lambda_{1}e^{-} \\[0.9em]
        0 &e^{-}\Lambda_{1}e^{-} 
    \end{pmatrix}
\end{equation*}

\smallskip

\noindent Then, there is an equivalence of triangulated categories

\smallskip

\begin{center}
    \begin{tikzcd}[column sep=1.5cm]
        \sgr{\T{\Lambda_{2}}}{\Z} \arrow[cramped]{r}[yshift=0.5mm]{-\underset{T\Lambda_{2}}{\overset{\mbox{\tiny \bf{L}}}{\otimes}}\mathscr{P}} &\sgr{\T{\Lambda_{1}}}{\Z}
    \end{tikzcd}
\end{center}

\smallskip

\noindent where ${\mathscr{P}}=\oplus_{n\in\Z} P(n)$ with $P=e^{+}\T{\Lambda_{1}} \oplus e^{-}\T{\Lambda_{1}}(-1)\in\per{\gr{\T{\Lambda_{1}}}{\Z}}$.
\end{prop} 

\medskip

Note that these two algebras have the same trivial extension. Assuming furthermore that they both are basic and have finite global dimension, they satisfy the assumptions of Corollary \ref{cor:non gradué projectif} for $\phi:\T{\Lambda_{1}}\rightarrow\T{\Lambda_{2}}$ being the identity and the equivalence of $\Z$-graded categories between the $\proj{\T{\Lambda_{i}}}{\Z}$ is given by 

\smallskip

\begin{equation*}
    \begin{aligned}
    \proj{\T{\Lambda_{2}}}{\Z} &\overset{\sim}{\longrightarrow} \proj{\T{\Lambda_{1}}}{\Z} \\[0.1cm]
    e^{-}\T{\Lambda_{2}} &\longmapsto e^{-}\T{\Lambda_{1}}(-1) \\
    e^{+}\T{\Lambda_{2}} &\longmapsto e^{+}\T{\Lambda_{1}}
    \end{aligned}
\end{equation*}

\smallskip

\noindent Hence, $\Lambda_{1}$ and $\Lambda_{2}$ are derived equivalent. In the following proposition, we show that we still have a derived equivalence between these two algebras if we only assume one of them to have finite global dimension. Moreover, we describe explicitly the tilting object associated to this derived equivalence.

\medskip

\begin{prop} \label{prop:tilting object for triangulated matrices}
    Let $\Lambda_{1}$ and $\Lambda_{2}$ be the triangular matrix algebras defined in Proposition \ref{prop:triangular matrices}. Assume that these two algebras are basic and that $\Lambda_{1}$ have finite global dimension. Then, ${T}=e^{+}\Lambda_{2}e^{+} \oplus D(e^{-}\Lambda_{2}e^{-})[1]$ is a tilting object in $\per{\Lambda_{2}}$ whose endomorphism algebra is isomorphic to $\Lambda_{1}$. In particular, we have an equivalence of triangulated categories

    \smallskip

    \begin{center}
        \begin{tikzcd}[column sep=1.5cm]
            \Db{\Lambda_{1}} \arrow[cramped]{r}[above]{-\underset{\Lambda_{1}}{\overset{\mbox{\tiny \bf{L}}}{\otimes}}T} &\Db{\Lambda_{2}}
        \end{tikzcd}
    \end{center}
    
\end{prop}

\medskip

\begin{proof}
    In \cite{Happel}, Happel proved that the functor from Theorem \ref{thm:Happel} always exists, is fully faithful and is an equivalence when the algebra have finite global dimension. Hence, we obtain the following commutative diagram

\smallskip

    \begin{center}
        \begin{tikzcd}[row sep=1.5cm, column sep=1.5cm]
            \Db{\Lambda_{2}} \arrow[cramped]{d}[left]{F_{2}} \arrow[cramped, dashed]{r}[above]{G}&\Db{\Lambda_{1}} \arrow[cramped]{d}[right]{F_{1}}[left]{\wr}  \\
            \sgr{\T{\Lambda_{2}}}{\Z} \arrow[cramped]{r}[yshift=0.5mm]{-\underset{T\Lambda_{2}}{\overset{\mbox{\tiny \bf{L}}}{\otimes}}\mathscr{P}}[below]{\sim} &\sgr{\T{\Lambda_{1}}}{\Z}
        \end{tikzcd}
    \end{center}

    \smallskip
    
    \noindent where $F_{i}:\Db{\Lambda_{i}}\rightarrow\sgr{\T{\Lambda_{i}}}{\Z}$ is the functor from Happel's theorem. Let us first prove that $G$ is an equivalence. Since the $F_{i}$ are fully faithful, it suffices to prove that $G$ is dense. Since $\Lambda_{1}$ have finite global dimension, it generates $\Db{\Lambda_{1}}$ as a triangulated category. Hence, it is enough to show the existence of $X\in\Db{\Lambda_{2}}$ such that $G(X)=\Lambda_{1}$. Let us prove that $T$ defined above satisfies such equality. In the following, we denote by $I=-\, {\overset{\mbox{\tiny \bf{L}}}{\otimes}}_{T\Lambda_{2}}\mathscr{P}$ and $\Sigma_{i}$ the suspension functor of the triangulated category $\sgr{\T{\Lambda_{i}}}{\Z}$. By construction of the functor $F_{2}$ given in \cite{Happel}

    \[F_{2}(T)=e^{+}\Lambda_{2}e^{+} \oplus \Sigma_{2} D(e^{-}\Lambda_{2}e^{-})\]

    \noindent since $e^{+}\Lambda_{2}e^{+},D(e^{-}\Lambda_{2}e^{-})\in\m{\Lambda_{2}}$. Moreover, 

    \[I\circ F_{2}(T)=e^{+}\Lambda_{2}e^{+} \oplus \Sigma_{1} D(e^{-}\Lambda_{2}e^{-})(-1)\]

    \noindent since $e^{+}\Lambda_{2}e^{+}\in\m{e^{+}\Lambda_{2}e^{+}}$ and $D(e^{-}\Lambda_{2}e^{-})\in\m{e^{-}\Lambda_{2}e^{-}}$. By definition of $\Lambda_{2}$, note that 
    
    \[e^{+}\Lambda_{2}e^{+}=e^{+}\Lambda_{1} \qquad \mbox{and} \qquad D(e^{-}\Lambda_{2}e^{-})=e^{-}D\Lambda_{1}\]

    \noindent Furthermore, using the following exact sequence in $\gr{\T{\Lambda_{1}}}{\Z}$

    \smallskip
    
    \begin{center}
        \begin{tikzcd}
            0 \arrow[cramped]{r} &e^{-}D\Lambda_{1}(-1) \arrow[cramped]{r} &e^{-}\T{\Lambda_{1}} \arrow[cramped]{r} &e^{-}\Lambda_{1} \arrow[cramped]{r} &0
        \end{tikzcd}
    \end{center}

    \smallskip

    \noindent we deduce that $\Sigma_{1}D(e^{-}\Lambda_{2}e^{-})(-1)$ is isomorphic to $e^{-}\Lambda_{1}$ in $\sgr{\T{\Lambda_{1}}}{\Z}$. Hence,

    \[I\circ F_{2}(T)=e^{+}\Lambda_{1} \oplus e^{-}\Lambda_{1}=\Lambda_{1}\]

    \noindent Using again the construction of the functor $F_{1}$ on $\m{\Lambda_{1}}$ given in \cite{Happel}, we conclude that $G(T)=\Lambda_{1}$. Hence, $G$ is indeed an equivalence and $T$ is indeed a tilting object yielding the desired equivalence of triangulated categories.
\end{proof}

\medskip

\begin{rem}
    It is not clear whether or not the diagram drawn in the proof of Proposition \ref{prop:tilting object for triangulated matrices} still commutes if we replace the functor $G$ with the equivalence defined in Proposition \ref{prop:tilting object for triangulated matrices}.
\end{rem}

\begin{rem}
    The tilting object of Proposition \ref{prop:tilting object for triangulated matrices} was already constructed in Theorem 4.9 of \cite{Ladkani}. Moreover, Ladkani proved that if the triangular matrix algebras both have infinite global dimension, they may not be derived equivalent \cite[Example 5.3]{Ladkani}.
\end{rem}

\medskip

\begin{ex} \label{ex:triangular matrix algebra}
Let us consider $\Lambda_{1}$ and $\Lambda_{2}$ two finite dimensional $k$-algebras defined by quiver and relations as follows

\smallskip

 \begin{center}
        $Q_{\Lambda_{1}}:$
        \begin{tikzcd}
            1 \arrow[cramped]{r}[above]{\alpha_{1}} \arrow[phantom]{r}[name=A]{}&2 \arrow[phantom]{r}[name=B]{} \arrow[cramped]{r}[above]{\alpha_{2}} &3 \arrow[dash, densely dashed, cramped, bend right=45, from=A]{B}
        \end{tikzcd} 
        \qquad \mbox{and} \qquad 
        $Q_{\Lambda_{2}}:$
        \begin{tikzcd}
            1 \arrow[cramped]{r}[above]{\beta_{1}} &2  &3 \arrow[cramped]{l}[above]{\beta_{2}}
        \end{tikzcd} 
    \end{center}

\smallskip

\noindent with relation $\alpha_{2}\alpha_{1}=0$ for $\Lambda_{1}$ and no relations for $\Lambda_{2}$. Note that these algebras satisfy the assumptions of Proposition \ref{prop:tilting object for triangulated matrices} where $e^{+}$ is the idempotent of $\Lambda_{1}$ corresponding to the vertices $1$ and $2$ and $e^{-}$ is the idempotent corresponding to the vertex $3$. Hence, $T=e^{+}\Lambda_{2}e^{+}\, \oplus\, D(e^{-}\Lambda_{2}e^{-})[1]$ is a tilting object in $\per{\Lambda_{2}}$ whose endomorphism algebra is isomorphic to $\Lambda_{1}$. Thus, $\Lambda_{1}$ and $\Lambda_{2}$ are derived equivalent. Here, 

\[T=\begin{tabular}{c} 2 \\ 1 \end{tabular} \, \oplus\, 1 \, \oplus\, 3[1]\]

\noindent Let us consider the Auslander-Reiten quiver of $\Db{\Lambda_{2}}$.

\smallskip

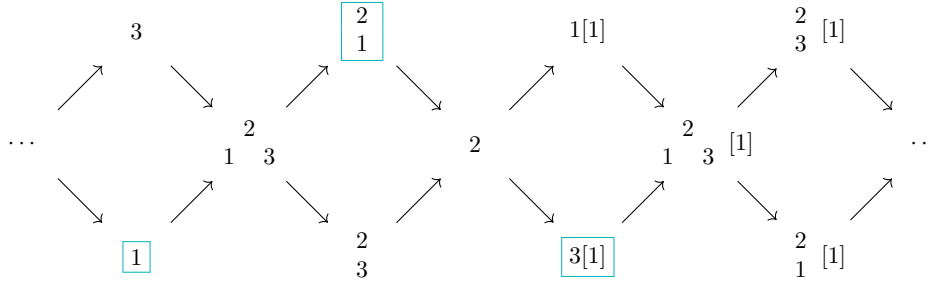
\begin{center}
    \begin{tikzpicture}[scale=0.75]
        \tikzstyle{point}=[minimum size=1cm, scale=0.9]
        \node[point] (0) at (0,0) {$\cdots$};
        \node[point] (1) at (2,2) {$3$};
        \node[point] (-1) at (2,-2) {\fcolorbox{Aquamarine}{white}{$1$}};
        \node[point] (2) at (4,0) {$\begin{tabular}{cc} \multicolumn{2}{c}{2} \\ 1 & 3 \end{tabular}$};
        \node[point] (3) at (6,2) {$\begin{tabular}{|c|} \arrayrulecolor{Aquamarine}\hline 2 \\ 1 \\ \hline \end{tabular}$};
        \node[point] (-3) at (6,-2) {$\begin{tabular}{c} 2 \\ 3 \end{tabular}$};
        \node[point] (4) at (8,0) {$2$};
        \node[point] (5) at (10,2) {$1[1]$};
        \node[point] (-5) at (10,-2) {\fcolorbox{Aquamarine}{white}{$3[1]$}};
        \node[point] (6) at (12,0) {$\begin{tabular}{cc} \multicolumn{2}{c}{2} \\ 1 & 3 \end{tabular}[1]$};
        \node[point] (7) at (14,2) {$\begin{tabular}{c} 2 \\ 3 \end{tabular}[1]$};
        \node[point] (-7) at (14,-2) {$\begin{tabular}{c} 2 \\ 1 \end{tabular}[1]$};
        \node[point] (8) at (16,0) {$\cdots$};
        \draw[->] (0)--(1);
        \draw[->] (1)--(2);
        \draw[->] (2)--(3);
        \draw[->] (3)--(4);
        \draw[->] (4)--(5);
        \draw[->] (5)--(6);
        \draw[->] (6)--(7);
        \draw[->] (7)--(8);
        \draw[->] (0)--(-1);
        \draw[->] (-1)--(2);
        \draw[->] (2)--(-3);
        \draw[->] (-3)--(4);
        \draw[->] (4)--(-5);
        \draw[->] (-5)--(6);
        \draw[->] (6)--(-7);
        \draw[->] (-7)--(8);
    \end{tikzpicture}
\end{center}

\smallskip

\noindent The blue rectangles in the previous quiver represent $T$. It is clear in this example that $T$ is indeed a tilting object in $\per{\Lambda_{2}}$ whose endomorphism algebra is isomorphic to $\Lambda_{1}$. 

\end{ex}

\medskip

In the following, we give a graded version of Proposition \ref{prop:tilting object for triangulated matrices}. We begin with constructing a $\Z$-grading on the triangular matrix algebras as before Corollary \ref{cor:gradué iso}. 

\medskip

Let $\Lambda_{1}$ and $\Lambda_{2}$ be finite dimensional $k$-algebras of the form 

\smallskip

\begin{equation*}
    \Lambda_{1}=\begin{pmatrix}
        e^{+}\Lambda_{1}e^{+} &0 \\[0.9em]
        e^{-}\Lambda_{1}e^{+} &e^{-}\Lambda_{1}e^{-}
    \end{pmatrix} \qquad \mbox{and} \qquad
    \Lambda_{2}=\begin{pmatrix}
        e^{+}\Lambda_{1}e^{+} &e^{+}D\Lambda_{1}e^{-} \\[0.9em]
        0 &e^{-}\Lambda_{1}e^{-} 
    \end{pmatrix}
\end{equation*}

\smallskip

\noindent Note that these two algebras have the same trivial extension and satisfy the conditions written before Corollary \ref{cor:gradué iso} for $\phi:\T{\Lambda_{1}}\rightarrow \T{\Lambda_{2}}$ being the identity and $B_{1}=B_{1}^{0} \cup B_{1}^{1}$ being a basis of $\Lambda_{1}$ where $B_{1}^{0}$ is a basis of $e^{+}\Lambda_{1}e^{+} \oplus e^{-}\Lambda_{1}e^{-}$ and $B_{1}^{1}$ is a basis of $e^{-}\Lambda_{1}e^{+}$. Thus, considering the indices modulo 2, the natural $\Z$-grading of $\T{\Lambda_{i+1}}$ induces a $\Z$-grading $d_{i}$ on $\Lambda_{i}$ as follows

\smallskip

\begin{equation*}
d_{1}(x)=\left\{ \begin{aligned} &0 &&\mbox{if $x\in e^{+}\Lambda_{1}e^{+} \oplus e^{-}\Lambda_{1}e^{-}$} \\
&1 &&\mbox{if $x\in e^{-}\Lambda_{1}e^{+}$}
\end{aligned}\right.
\qquad \mbox{and} \qquad 
d_{2}(x)=\left\{ \begin{aligned} &0 &&\mbox{if $x\in e^{+}\Lambda_{1}e^{+} \oplus e^{-}\Lambda_{1}e^{-}$} \\
&1 &&\mbox{if $x\in e^{+}D\Lambda_{1}e^{-}$}
\end{aligned}\right.
\end{equation*}

\smallskip

\noindent This induces a $\Z^{2}$-grading on $\T{\Lambda_{i}}$ where the first component represents the natural $\Z$-grading of the trivial extension. Moreover, since the following isomorphism 

\smallskip

\begin{equation*}
    \begin{aligned}
        \theta : &&\Z^{2} &&&\longrightarrow &&\Z^{2} \\
        &&(x,y) &&&\longmapsto && (x+y,-y)
    \end{aligned}
\end{equation*}

\smallskip

\noindent sends the $\Z^{2}$-grading of $\T{\Lambda_{1}}$ onto the $\Z^{2}$-grading of $\T{\Lambda_{2}}$, there is a triangle equivalence between the stable categories $\sgr{\T{\Lambda_{i}}}{\Z^{2}}$. Hence, assuming that $\Lambda_{1}$ and $\Lambda_{2}$ are basic and of finite global dimension, we obtain a derived equivalence between $\gr{\Lambda_{1}}{\Z}$ and $\gr{\Lambda_{2}}{\Z}$ thanks to Theorem \ref{thm:Happel gradué}. Again, we show in the following result that we only need one of them to have finite global dimension for the derived equivalence to hold. Moreover, we describe in this case the tilting object associated to this derived equivalence.

\medskip

\begin{prop} \label{prop:tilting object for graded triangular matrices}
Let $(\Lambda_{1},d_{1})$ and $(\Lambda_{2},d_{2})$ be the $\Z$-graded triangular matrix algebras defined previously. Assume that these two algebras are basic and that $\Lambda_{1}$ have finite global dimension. We denote by $\tau_{i}$ the Auslander-Reiten translation in $\Db{\gr{\Lambda_{i}}{\Z}}$. Then, 

\smallskip

\[{\mathscr{T}}=\bigoplus_{n\in\Z} \ (\tau_{2})^{n}T[2n](-n) \qquad \mbox{with} \qquad T=e^{+}\Lambda_{2}e^{+}\oplus \tau_{2}^{-1}D(e^{-}\Lambda_{2}e^{-})[-1]\]

\smallskip

\noindent is a tilting object in $\per{\gr{\Lambda_{2}}{\Z}}$ satisfying that for all $n\in\Z$, there is an isomorphism of $\Z$-graded $\Lambda_{1}$-modules

\[\bigoplus_{m\in\Z}\Hom[\per{\gr{\Lambda_{2}}{\Z}}]{(\tau_{2})^{m}T[2m](-m)}{(\tau_{2})^{n}T[2n](-n)}\simeq \Lambda_{1}(n)\]

\noindent In particular, we have an equivalence of triangulated categories 

\smallskip

\begin{center}
    \begin{tikzcd}[column sep=1.5cm]
        \Db{\gr{\Lambda_{1}}{\Z}} \arrow[cramped]{r}[above]{-\underset{\mathrm{gr}(\Lambda_{1})}{\overset{\mbox{\tiny \bf{L}}}{\otimes}}\mathscr{T}} &\Db{\gr{\Lambda_{2}}{\Z}}
    \end{tikzcd}
\end{center}

\smallskip

\noindent where $-\otimes_{\mathrm{gr}(\Lambda_{1})}{\mathscr{T}}:\mathrm{K^{b}}(\gr{\Lambda_{1}}{\Z})\rightarrow \mathrm{K^{b}}(\gr{\Lambda_{2}}{\Z})$ is the graded tensor product between the bounded homotopy categories of $\gr{\Lambda_{i}}{\Z}$ as defined in \cite{ZH}.

\end{prop}

\medskip

\begin{proof}
    Thanks to the graded version of Happel's theorem \cite{thèseSoto}, the functor from Theorem \ref{thm:Happel gradué} always exists and is fully faithful. Moreover, it is an equivalence when the algebra is of finite global dimension. Hence, we obtain the following commutative diagram
    
    \smallskip

    \begin{center}
        \begin{tikzcd}[row sep=1.5cm, column sep=1.5cm]
            \Db{\gr{\Lambda_{2}}{\Z}} \arrow[cramped, dashed]{r}[above]{G} \arrow[cramped]{d}[left]{F_{2}} &\Db{\gr{\Lambda_{1}}{\Z}} \arrow[cramped]{d}[left]{\wr}[right]{F_{1}}  \\
            \sgr{\T{\Lambda_{2}}}{\Z^{2}} \arrow[cramped]{r}[above]{I}[below]{\sim} &\sgr{\T{\Lambda_{1}}}{\Z^{2}}
        \end{tikzcd}
    \end{center}

    \smallskip

    \noindent where $F_{i}:\Db{\gr{\Lambda_{i}}{\Z}}\rightarrow\sgr{T\Lambda_{i}}{\Z^{2}}$ is the functor of the graded version of Happel's theorem and $I:\sgr{\T{\Lambda_{1}}}{\Z^{2}} \rightarrow\sgr{\T{\Lambda_{2}}}{\Z^{2}}$ is the triangle equivalence obtained thanks to the isomorphism $\theta:\Z^{2}\rightarrow \Z^{2}$ defined previously. Let us first prove that $G$ is an equivalence. With similar arguments as in Proposition \ref{prop:tilting object for triangulated matrices}, it suffices to prove that for all $n\in\Z$, there exists $_{n}X\in\Db{\gr{\Lambda_{2}}{\Z}}$ such that $G(_{n}X)=\Lambda_{1}(n)$. Let us prove that 
    
    \[_{n}T:=(\tau_{2})^{n}T[2n](-n)\]
    
    \noindent satisfies this equality for all $n\in\Z$. We denote by $\Sigma_{i}$ the suspension functor of the triangulated category $\sgr{T\Lambda_{i}}{\Z^{2}}$. By construction of the functor $F_{2}$, 

    \[F_{2}(_{n}T)=e^{+}\Lambda_{2}e^{+}(n,-n)\oplus \Sigma_{2}D(e^{-}\Lambda_{2}e^{-})(n-1,-n)\]

    \noindent since $e^{+}\Lambda e^{+}, D(e^{-}\Lambda_{2}e^{-})\in\gr{\Lambda_{2}}{\Z}$. Hence, by definition of $I$

    \[I\circ F_{2}(_{n}T)=e^{+}\Lambda_{2}e^{+}(0,n) \oplus \Sigma_{1}D(e^{-}\Lambda_{2}e^{-})(-1,n)\]

    \noindent Note that $e^{+}\Lambda_{2}e^{+}=e^{+}\Lambda_{1}$ and $D(e^{-}\Lambda_{2}e^{-})=e^{-}D\Lambda_{1}$. Furthermore, using the following exact sequence in $\gr{\T{\Lambda_{1}}}{\Z^{2}}$

    \smallskip
    
    \begin{center}
        \begin{tikzcd}
            0 \arrow[cramped]{r} &e^{-}D\Lambda_{1}(-1,n) \arrow[cramped]{r} &e^{-}T\Lambda_{1}(0,n) \arrow[cramped]{r} &e^{-}\Lambda_{1}(0,n) \arrow[cramped]{r} &0
        \end{tikzcd}
    \end{center}

    \smallskip

    \noindent we deduce that $\Sigma_{1}D(e^{-}\Lambda_{2}e^{-})(-1,n)$ is isomorphic to $e^{-}\Lambda_{1}(0,n)$ in $\sgr{\T{\Lambda_{1}}}{\Z^{2}}$. To sum up, we have

    \[I\circ F_{2}(_{n}T)=e^{+}\Lambda_{1}(0,n)\oplus e^{-}\Lambda_{1}(0,n)=\Lambda_{1}(0,n)\]

    \noindent Using again the construction of the functor $F_{1}$ on $\gr{\Lambda_{2}}{\Z}$, we conclude that $G(_{n}T)=\Lambda_{1}(n)$. Hence, we conclude that $G$ is indeed an equivalence. Moreover, $\mathscr{T}$ is indeed a tilting object yielding the desired equivalence of triangulated categories thanks to Theorem 1 of \cite{ZH}.    
    \end{proof}

\medskip

\section{Application to gentle algebras}

\medskip

We apply the results of the previous section in the particular case of gentle algebras. Throughout this section, we assume that the field $k$ is algebraically closed. Let us recall the definition of a gentle algebra.

\medskip

\begin{defn} \label{def:gentle algebras}
    A finite dimensional algebra $A$ is said to be \textit{gentle} if it is isomorphic to an algebra of the form $kQ/I$ where

    \smallskip

    \begin{itemize}[label=\textbullet, font=\tiny]
    \item $Q$ is a finite quiver whose vertices are the source of at most two arrows and the target of at most two arrows;
    \item $I$ is an admissible ideal (i.e. $I$ contains only paths of length at least 2 and contains all the paths of length at least $m$ for some $m\ge 2$) and is generated by paths of length 2;
    \item for every arrow $\alpha$ in $Q$, there is at most one arrow $\beta$ such that $\beta\alpha\in I$, at most one arrow $\beta'$ such that $\alpha\beta'\in I$, at most one arrow $\gamma$ such that $\gamma\alpha\notin I$ and at most one arrow $\gamma'$ such that $\alpha\gamma'\notin I$.
    \end{itemize}
\end{defn}

\medskip

We are interested in these algebras since their trivial extensions are known to be Brauer graph algebras of multiplicity identically 1 \cite{Schroll2}. 

\medskip

\subsection{Brauer graph algebras as trivial extensions}

\medskip

We recall the definition of Brauer graph algebras of multiplicity identically 1 and their link with gentle algebras via the trivial extension. We refer to \cite{Schroll} for more information about Brauer graph algebras. 

\medskip

\begin{defn} \label{def:Brauer graph}
    A \textit{Brauer graph (of multiplicity 1)} is a tuple $\Gamma=(\Gamma_{0},\Gamma_{1},\circ)$ where 

    \smallskip

    \begin{itemize}[label=\textbullet, font=\tiny]
        \item $(\Gamma_{0},\Gamma_{1})$ is a finite graph whose vertex set is $\Gamma_{0}$ and edge set is $\Gamma_{1}$;
        \item $\circ$ is called the \textit{orientation} of $\Gamma$ and is given by a cyclic ordering of the edges incident with $v$ for every vertex $v\in \Gamma_{0}$. Note that the cyclic ordering of a vertex incident to a single edge $i$ is given by $i$.
        
    \end{itemize}
\end{defn}

\medskip

Unless otherwise stated, the orientation of a Brauer graph will be given by locally embedding each vertex into the counterclockwise oriented plane.

\medskip

\begin{ex} \label{ex:Brauer graph} Let $\Gamma$ be the Brauer graph given by

\smallskip

    
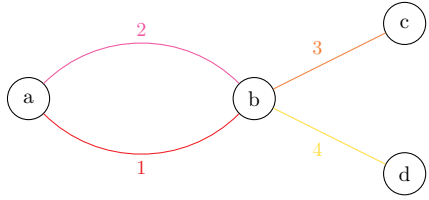
\begin{figure}[H]
        \centering
        \begin{tikzpicture}[scale=1]
        \tikzstyle{vertex}=[minimum size=0.7cm,circle,draw,scale=0.8]
        \node[vertex] (a) at (-1.5,0) {a};
        \node[vertex] (b) at (1.5,0) {b};
        \node[vertex] (c) at (3.5,1) {c};
        \node[vertex] (d) at (3.5,-1) {d};
        \draw[Red] (a) to[bend right=45] node[midway, below, scale=0.75]{1} (b);
        \draw[VioletRed] (a) to[bend left=45] node[midway, above, scale=0.75]{2} (b); 
        \draw[Orange] (b)--(c) node[midway, above left, scale=0.75]{3};
        \draw[Goldenrod] (b)--(d) node[midway, below left , scale=0.75]{4};
        \end{tikzpicture}
        \label{Brauer graph}
    \end{figure}

\smallskip

\noindent Using our previous convention, the cyclic ordering at vertex $a$ is given by $1<2<1$, at vertex $b$ by $1<4<3<2<1$, at vertex $c$ by $3$ and at vertex $d$ by $4$.
\end{ex}

\medskip

\begin{defn} \label{def:Brauer graph algebra}
The \textit{Brauer graph algebra} $B_{\Gamma}$ associated to a Brauer graph $\Gamma$ is the path algebra $kQ_{\Gamma}/I_{\Gamma}$ where 
    
\smallskip

\begin{itemize}[label=\textbullet, font=\tiny]
    \item The vertices of $Q_{\Gamma}$ are given by the edges of $\Gamma$. Moreover, given two edges $i$ and $j$ of $\Gamma$ incident with the same vertex $v$ in $\Gamma$, there is an arrow in $Q_{\Gamma}$ from $i$ to $j$ if $j$ is the direct successor of $i$ in the cyclic ordering at $v$. In particular, for every vertex $v$ in $\Gamma$ that is not incident to a single edge, the cyclic ordering at $v$ induces an oriented cycle $C_{v}^{i}$ in $Q_{\Gamma}$ that begins and ends with $i$ for each edge $i$ in $\Gamma$ that is incident with $v$.
    
    \item The ideal $I_{\Gamma}$ is generated by three types of relations

    \smallskip

    \begin{enumerate}[label=(\Roman*)]
    \item \label{item1:Brauer graph algebra}
    
    \[C_{v}^{i}-C_{v'}^{i}\]
        
    for any $v$ and $v'$ that are not incident to a single edge and for any edge $i$ in $\Gamma$ that is incident with $v$ and $v'$. Note that $v$ and $v'$ may coincide and it happens exactly when $i$ is a loop in $\Gamma$.

    \item \label{item2:Brauer graph algebra}

    \[\alpha_{1}C_{v}^{i}\]

    for any $v$ that is not incident to a single edge and for any edge $i$ in $\Gamma$ that is incident with $v$ where $C_{v}^{i}=\alpha_{n}\ldots\alpha_{1}$.

    \item \label{item3:Brauer graph algebra}

    \[\beta\alpha\]

    for any arrows $\beta,\alpha$ in $Q_{\Gamma}$ such that $\beta\alpha$ is not a subpath of any $C_{v}^{i}$.
    \end{enumerate}
    \end{itemize}
\end{defn}

\medskip

In general, the relations defining a Brauer graph algebra are not minimal as we can see in the following example.

\medskip

\begin{ex} \label{ex:quiver and relations of a Brauer graph algebra}
    Let $\Gamma$ be the Brauer graph defined in Example \ref{ex:Brauer graph}. By construction, its associated Brauer graph algebra $B$ is given by $kQ/I$ whose quiver $Q$ is the following 
    
    \smallskip

    
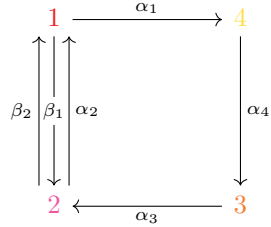
\begin{figure}[H]
        \centering
        \begin{tikzcd}[column sep=2cm, row sep=2cm,scale=0.75]
            \textcolor{Red}{1} \arrow[cramped]{r}[above]{\alpha_{1}} \arrow[cramped]{d}[description]{\beta_{1}} & \textcolor{Goldenrod}{4} \arrow[cramped]{d}[right]{\alpha_{4}} \\
            \textcolor{VioletRed}{2} \arrow[cramped, xshift=2mm]{u}[right]{\alpha_{2}} \arrow[cramped, xshift=-2mm]{u}[left]{\beta_{2}}   & \textcolor{Orange}{3} \arrow[cramped]{l}[below]{\alpha_{3}} 
        \end{tikzcd}
        \label{Quiver of a Brauer graph}
    \end{figure}

    \smallskip

    \noindent and whose ideal of relations $I$ is generated by 

    \smallskip

    \begin{enumerate}[label=(\Roman*)]
    \item \label{I} $\beta_{2}\beta_{1}-\alpha_{2}\alpha_{3}\alpha_{4}\alpha_{1}$, $\beta_{1}\beta_{2}-\alpha_{3}\alpha_{4}\alpha_{1}\alpha_{2}$ ;
    \item \label{II} $\beta_{1}\beta_{2}\beta_{1}$, $\beta_{2}\beta_{1}\beta_{2}$, $\alpha_{1}\alpha_{2}\alpha_{3}\alpha_{4}\alpha_{1}$, $\alpha_{2}\alpha_{3}\alpha_{4}\alpha_{1}\alpha_{2}$, $\alpha_{3}\alpha_{4}\alpha_{1}\alpha_{2}\alpha_{3}$, $\alpha_{4}\alpha_{1}\alpha_{2}\alpha_{3}\alpha_{4}$;
    \item \label{III} $\beta_{2}\alpha_{3}$, $\alpha_{1}\beta_{2}$, $\beta_{1}\alpha_{2}$, $\alpha_{2}\beta_{1}$.
    \end{enumerate}

    \smallskip

    \noindent In this case all the relations \ref{II} except $\alpha_{4}\alpha_{1}\alpha_{2}\alpha_{3}\alpha_{4}$ are determined by \ref{I} and \ref{III}.
\end{ex}

\medskip

Let us recall now how Brauer graph algebras are linked to gentle algebras through the trivial extension.

\medskip

\begin{defn} \label{def:admissible cut}
    Let $Q$ be a quiver associated to a Brauer graph $\Gamma$. An \textit{admissible cut} $\Delta$ of $Q$ is a set of arrows of $Q$ containing exactly one arrow in each oriented cycle (up to permutation) induced by the cyclic ordering at a vertex $v$ of $\Gamma$ that is not incident to a single edge.
\end{defn}

\medskip 

\begin{thm}[Schroll {\cite[Theorem 3.13 and Corollary 3.14]{Schroll}}] \label{thm:Schroll}
    Let $B=kQ/I$ be a Brauer graph algebra of Brauer graph $\Gamma$ and $\Delta$ be an admissible cut of $Q$. Then

    \smallskip

    \begin{itemize}[label=\textbullet, font=\tiny]
        \item The path algebra $B_{\Delta}:=kQ/<I\cup \Delta>$ is a gentle algebra, where $<I\cup \Delta>$ is the ideal of $kQ$ generated by $I\cup \Delta$.

        \item $B$ is isomorphic to the trivial extension of $B_{\Delta}$.
    \end{itemize}

    \smallskip

    \noindent Conversely, for every gentle algebra $\Lambda$, 
    there exists a unique Brauer graph algebra $B=kQ/I$ such that $\Lambda=B_{\Delta}$ for some admissible cut $\Delta$ of $Q$. Moreover, this Brauer graph algebra is isomorphic to the trivial extension of $\Lambda$.
\end{thm}

\medskip

Notice that the choice of an admissible cut $\Delta$ in a Brauer graph algebra $B$ gives a $\Z$-grading on its associated quiver $Q$ : the arrows in $\Delta$ are of degree 1 and the others of degree 0. Since the relations \ref{item1:Brauer graph algebra}, \ref{item2:Brauer graph algebra} and \ref{item3:Brauer graph algebra} generating $I$ are homogeneous with respect to this $\Z$-grading on $Q$, this induces a $\Z$-grading on $B$. We denote $(B,\Delta)$ the associated $\Z$-graded algebra. In particular, since this $\Z$-grading coincide with the natural $\Z$-grading of the trivial extension, $\Lambda=B_{\Delta}$ may be seen as the degree 0 subalgebra of $(B,\Delta)$. The following corollary is a reformulation of Corollary \ref{cor:non gradué projectif} for this setting, where $\phi$ is the identity between the trivial extensions.

\medskip

\begin{cor} \label{cor:non gradué BGA}
    Let $B$ be a Brauer graph algebra with two admissible cuts $\Delta_{1}$ and $\Delta_{2}$ such that the gentle algebras $\Lambda_{i}:=B_{\Delta_{i}}$ have finite global dimension. We denote by $\proj{B,\Delta_{i}}{\Z}$ the full subcategory of $\gr{B,\Delta_{i}}{\Z}$ whose objects are the projectives. Let us assume that for all $j$ there exists $n_{j}\in \Z$ such that 

    \smallskip
    
    \begin{equation*}
    \begin{aligned}
    \proj{B,\Delta_{1}}{\Z} &\overset{\sim}{\longrightarrow} \proj{B,\Delta_{2}}{\Z} \\[0.1cm]
    e_{j}B &\longmapsto e_{j}B(n_{j})
    \end{aligned}
    \end{equation*}

    \smallskip
    
    \noindent gives an equivalence of $\Z$-graded categories. Then, the algebras $\Lambda_{1}$ and $\Lambda_{2}$ are derived equivalent.

    \smallskip

\end{cor}

\medskip

\begin{ex} \label{ex:non gradué BGA}
   Let us consider the Brauer graph algebra $B=kQ/I$ defined in Example \ref{ex:quiver and relations of a Brauer graph algebra} with two admissible cuts $\Delta_{1}=\{\alpha_{2},\beta_{2}\}$ and $\Delta_{2}=\{\beta_{1}, \alpha_{1}\}$. In this case, the $\Z$-gradings of $B$ induced by the $\Delta_{i}$ are given on $Q$ by 

    \smallskip

        
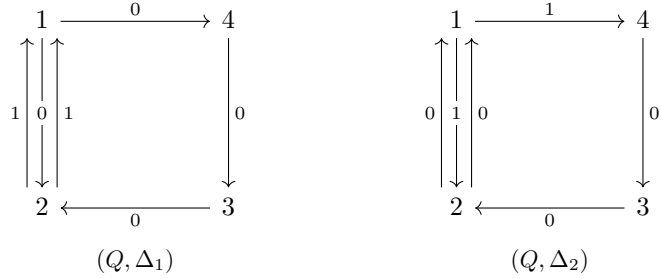
\begin{figure}[H]   
        \centering
        \begin{tikzcd}[column sep=2cm, row sep=2cm, scale=0.75]
            1 \arrow[cramped]{r}[above]{0} \arrow[cramped]{d}[description]{0} & 4 \arrow[cramped]{d}[right]{0} \\
            2 \arrow[cramped, xshift=2mm]{u}[right]{1} \arrow[cramped, xshift=-2mm]{u}[left]{1}   & 3 \arrow[cramped]{l}[below]{0} \arrow[phantom]{l}[yshift=-7mm, scale=0.9]{(Q,\Delta_{1})}
        \end{tikzcd}
        \hspace{2cm}
        \begin{tikzcd}[column sep=2cm, row sep=2cm, scale=0.75]
            1 \arrow[cramped]{r}[above]{1} \arrow[cramped]{d}[description]{1} & 4 \arrow[cramped]{d}[right]{0} \\
            2 \arrow[cramped, xshift=2mm]{u}[right]{0} \arrow[cramped, xshift=-2mm]{u}[left]{0}   & 3 \arrow[cramped]{l}[below]{0} \arrow[phantom]{l}[yshift=-7mm, scale=0.9]{(Q,\Delta_{2})}
        \end{tikzcd}
    \end{figure}

    \smallskip

    \noindent By definition, the gentle algebras $\Lambda_{i}:=B_{\Delta_{i}}$ are the degree 0 part of $(B,\Delta_{i})$. In particular, their quivers are given by

    \smallskip

    
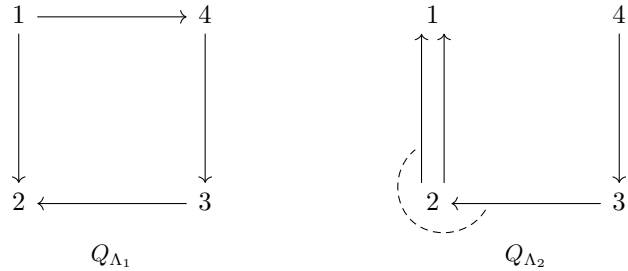
\begin{center}
        \begin{tikzcd}[column sep=2cm, row sep=2cm, scale=0.75]
        1 \arrow[cramped]{r} \arrow[cramped]{d} & 4 \arrow[cramped]{d} \\
            2   & 3 \arrow[cramped]{l} \arrow[phantom]{l}[yshift=-7mm, scale=0.9]{Q_{\Lambda_{1}}}
        \end{tikzcd}
        \hspace{2cm}
        \begin{tikzcd}[column sep=2cm, row sep=2cm, scale=0.75]
            1  & 4 \arrow[cramped]{d} \\
            2 \arrow[cramped, xshift=1.5mm]{u} \arrow[phantom, xshift=-1.5mm]{u}[near start, name=B]{}[very near start, name=C]{} \arrow[cramped, xshift=-1.5mm]{u}   & 3 \arrow[cramped]{l} \arrow[phantom]{l}[near end, name=A]{} \arrow[phantom]{l}[yshift=-7mm, scale=0.9]{Q_{\Lambda_{2}}}
            \arrow[dash, densely dashed, to path={(A)to[bend left=50](-2.01,-2.01)to[bend left=50](B)}]
        \end{tikzcd}
    \end{center}
    
    \smallskip
    
    \noindent Thus, we have the following equivalence of $\Z$-graded categories

    \smallskip
    
    \begin{equation*}
    \begin{aligned}
    \proj{B,\Delta_{1}}{\Z} &\overset{\sim}{\longrightarrow} \proj{B,\Delta_{2}}{\Z} \\[0.1cm]
    e_{1}B &\longmapsto e_{1}B(1) &&\\
    e_{j}B &\longmapsto e_{j}B &&\mbox{for $j=2,3,4$}
    \end{aligned}
    \end{equation*}

    \smallskip

    \noindent Note that the algebras $\Lambda_{i}$ both have finite global dimension. Hence, they are derived equivalent thanks to Corollary \ref{cor:non gradué BGA}.

\end{ex}

\medskip

Let $B=kQ/I$ be a Brauer graph algebra with two admissible cuts $\Delta_{1}$ and $\Delta_{2}$. We have seen that each admissible cut induces a $\Z$-grading on the arrows of $Q$ and the relations \ref{item1:Brauer graph algebra}, \ref{item2:Brauer graph algebra} and \ref{item3:Brauer graph algebra} generating $I$ are homogeneous with respect to each $\Z$-grading. Thus, we obtain a $\Z^{2}$-grading on $B$ such that the $i$-th coordinate corresponds to the $\Z$-grading induced by $\Delta_{i}$. We denote by $(B,\Delta_{1},\Delta_{2})$ the associated $\Z^{2}$-graded algebra. Let $\Lambda_{i}:=B_{\Delta_{i}}$ be the gentle algebra corresponding to the admissible cut $\Delta_{i}$. Considering the indices modulo 2, note that $\Delta_{i+1}$ induces a $\Z$-grading on the arrows of $\Lambda_{i}$. This $\Z$-grading extends to a $\Z$-grading on $\Lambda_{i}$ since the relations of the ideal generated by $I\cup\Delta_{i}$ are homogeneous with respect to this $\Z$-grading.
We denote by $(\Lambda_{1}, \Delta_{2})$ and $(\Lambda_{2}, \Delta_{1})$ these $\Z$-graded algebras. Thus, one can construct a $\Z^{2}$-grading on $\T{\Lambda_{i}}$ where the first component is the natural $\Z$-grading of the trivial extension. This leads to the following corollary.

\medskip

\begin{cor} \label{cor:gradué BGA}
    Let $B$ be a Brauer graph algebra with two admissible cuts $\Delta_{1}$ and $\Delta_{2}$ such that the gentle algebras $\Lambda_{i}:=B_{\Delta_{i}}$ have finite global dimension. Then, we have an equivalence of triangulated categories

    \smallskip
    
    \[\Db{\gr{\Lambda_{1},\Delta_{2}}{\Z}}\simeq \Db{\gr{\Lambda_{2},\Delta_{1}}{\Z}}\]
\end{cor}

\medskip

\begin{proof}
    This follows from the fact that the following isomorphism 

    \smallskip
    
\begin{equation*}
    \begin{aligned}
        &&\Z^{2} &&&\longrightarrow &&\Z^{2} \\
        &&(x,y) &&&\longmapsto && (x+y,-y)
    \end{aligned}
\end{equation*}

\smallskip

\noindent sends the $\Z^{2}$-grading of $\T{\Lambda_{1}}$ onto the $\Z^{2}$-grading of $\T{\Lambda_{2}}$.
\end{proof}

\medskip

\begin{rem} \label{rem:hypothèse cor gradué iso}
    One can check that $\Lambda_{1}$ and $\Lambda_{2}$ defined in the previous corollary satisfy the assumptions written before Corollary \ref{cor:gradué iso} for $\phi:\T{\Lambda_{1}}\rightarrow \T{\Lambda_{2}}$ being the identity.
\end{rem}

\medskip

\begin{ex} \label{ex:gradué BGA}
    Let us consider the Brauer graph algebra $B=kQ/I$ defined in Example \ref{ex:quiver and relations of a Brauer graph algebra} with two admissible cuts $\Delta_{1}=\{\alpha_{2},\beta_{2}\}$ and $\Delta_{2}=\{\alpha_{2},\beta_{1}\}$. In this case, the $\Z$-gradings of $B$ induced by the $\Delta_{i}$ are given on $Q$ by
    
    
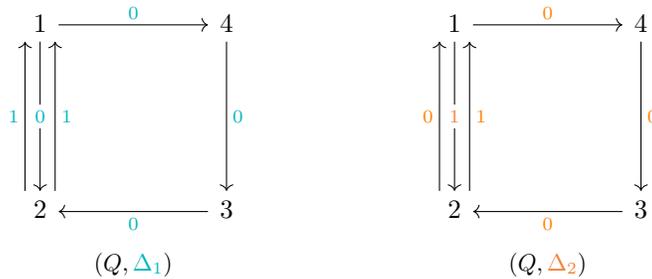
\begin{figure}[H]   
        \centering
        \begin{tikzcd}[column sep=2cm, row sep=2cm, scale=0.75]
            1 \arrow[cramped]{r}[above, Aquamarine]{0} \arrow[cramped]{d}[description]{\textcolor{Aquamarine}{0}} & 4 \arrow[cramped]{d}[right, Aquamarine]{0} \\
            2 \arrow[cramped, xshift=2mm]{u}[right, Aquamarine]{1} \arrow[cramped, xshift=-2mm]{u}[left, Aquamarine]{1}   & 3 \arrow[cramped]{l}[below, Aquamarine]{0} \arrow[phantom]{l}[yshift=-7mm, scale=0.9]{(Q,\textcolor{Aquamarine}{\Delta_{1}})}
        \end{tikzcd}
        \hspace{2cm}
        \begin{tikzcd}[column sep=2cm, row sep=2cm, scale=0.75]
            1 \arrow[cramped]{r}[above,orange]{0} \arrow[cramped]{d}[description]{\textcolor{Orange}{1}} & 4 \arrow[cramped]{d}[right,orange]{0} \\
            2 \arrow[cramped, xshift=2mm]{u}[right,orange]{1} \arrow[cramped, xshift=-2mm]{u}[left,orange]{0}   & 3 \arrow[cramped]{l}[below,orange]{0} \arrow[phantom]{l}[yshift=-7mm, scale=0.9]{(Q,\textcolor{Orange}{\Delta_{2}})}
        \end{tikzcd}
    \end{figure}
    
    \noindent By definition, the gentle algebras $\Lambda_{i}:=B_{\Delta_{i}}$ are the degree 0 part of $(B,\Delta_{i})$. Moreover, the $\Z$-grading of $(B,\Delta_{i+1})$ gives rise to a $\Z$-grading on $\Lambda_{i}$ which is given on its quiver as follows
    
    \smallskip
    
    
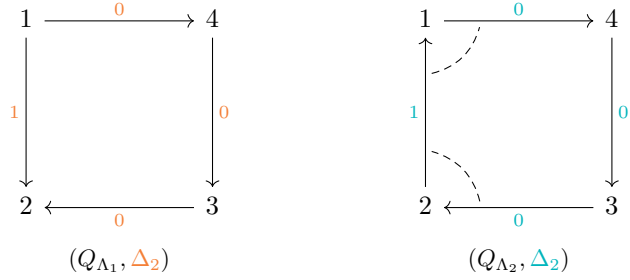
\begin{figure}[H]
        \centering
       \begin{tikzcd}[column sep=2cm, row sep=2cm, scale=0.75]
            1 \arrow[cramped]{r}[above,Orange]{0} \arrow[cramped]{d}[left]{\textcolor{Orange}{1}} & 4 \arrow[cramped]{d}[right, Orange]{0} \\
            2    & 3 \arrow[cramped]{l}[below, Orange]{0} \arrow[phantom]{l}[yshift=-7mm, scale=0.9]{(Q_{\Lambda_{1}},\textcolor{Orange}{\Delta_{2}})}
        \end{tikzcd}
        \hspace{2cm} 
        \begin{tikzcd}[column sep=2cm, row sep=2cm, scale=0.75]
            1 \arrow[cramped]{r}[above,Aquamarine]{0} \arrow[phantom]{r}[near start, name=D]{} & 4 \arrow[cramped]{d}[right,Aquamarine]{0} \\
            2  \arrow[cramped]{u}[left,Aquamarine]{1}   \arrow[phantom]{u}[near start, name=B]{}[near end, name=C]{} & 3 \arrow[cramped]{l}[below,Aquamarine]{0} \arrow[phantom]{l}[yshift=-7mm, scale=0.9]{(Q_{\Lambda_{2}},\textcolor{Aquamarine}{\Delta_{2}})}
            \arrow[phantom]{l}[near end, name=A]{} \arrow[dash, densely dashed, cramped, bend right, cramped, from=A]{B} \arrow[dash, densely dashed, cramped, bend right, cramped, from=C]{D}
         \end{tikzcd}
    \end{figure}
    
    \smallskip
    
    \noindent Thus, one can construct a $\Z^{2}$-grading on $\T{\Lambda_{i}}=B$ whose first component is given by the natural $\Z$-grading of the trivial extension. The $\Z^{2}$-graded quiver of $\T{\Lambda_{i}}$ is given by 

     \smallskip

    
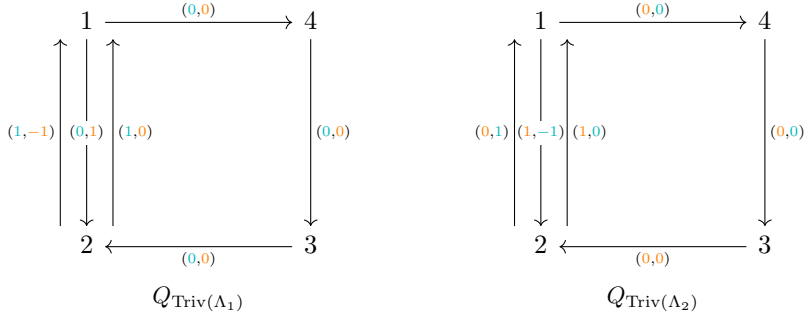
\begin{figure}[H]
        \centering
        \begin{tikzcd}[column sep=2.5cm, row sep=2.5cm, scale=0.35]
            1 \arrow[cramped]{r}[above,scale=0.8]{(\textcolor{Aquamarine}{0},\textcolor{orange}{0})} \arrow[cramped]{d}[description, scale=0.8]{(\textcolor{Aquamarine}{0},\textcolor{orange}{1})} & 4 \arrow[cramped]{d}[right,scale=0.8]{(\textcolor{Aquamarine}{0},\textcolor{orange}{0})} \\
            2 \arrow[cramped, xshift=3.5mm]{u}[right,scale=0.8]{(\textcolor{Aquamarine}{1},\textcolor{orange}{0})} \arrow[cramped, xshift=-3.5mm]{u}[left,scale=0.8]{(\textcolor{Aquamarine}{1},\textcolor{orange}{-1})}   & 3 \arrow[cramped]{l}[below,scale=0.8]{(\textcolor{Aquamarine}{0},\textcolor{orange}{0})} \arrow[phantom]{l}[yshift=-7mm, scale=0.9]{Q_{\T{\Lambda_{1}}}}
        \end{tikzcd}
        \qquad \qquad 
        \begin{tikzcd}[column sep=2.5cm, row sep=2.5cm, scale=0.35]
            1 \arrow[cramped]{r}[above,scale=0.8]{(\textcolor{orange}{0},\textcolor{Aquamarine}{0})} \arrow[cramped]{d}[description, scale=0.8]{(\textcolor{orange}{1},\textcolor{Aquamarine}{-1})} & 4 \arrow[cramped]{d}[right,scale=0.8]{(\textcolor{orange}{0},\textcolor{Aquamarine}{0})} \\
            2 \arrow[cramped, xshift=3.5mm]{u}[right,scale=0.8]{(\textcolor{orange}{1},\textcolor{Aquamarine}{0})} \arrow[cramped, xshift=-3.5mm]{u}[left,scale=0.8]{(\textcolor{orange}{0},\textcolor{Aquamarine}{1})}   & 3 \arrow[cramped]{l}[below,scale=0.8]{(\textcolor{orange}{0},\textcolor{orange}{0})} \arrow[phantom]{l}[yshift=-7mm, scale=0.9]{Q_{\T{\Lambda_{2}}}}
        \end{tikzcd}
    \end{figure}

    \noindent Since the following isomorphism 

    \smallskip
    
    \begin{equation*}
    \begin{aligned}
        &&\Z^{2} &&&\longrightarrow &&\Z^{2} \\
        &&(x,y) &&&\longmapsto && (x+y,-y)
    \end{aligned}
\end{equation*}

\smallskip

\noindent sends the $\Z^{2}$-grading of $\T{\Lambda_{1}}$ onto the $\Z^{2}$-grading of $\T{\Lambda_{2}}$, we conclude that there is an equivalence of triangulated categories 

\[\Db{\gr{\Lambda_{1},\Delta_{2}}{\Z}}\simeq\Db{\gr{\Lambda_{2},\Delta_{1}}{\Z}}\]

\noindent However, note that the algebras $\Lambda_{1}$ and $\Lambda_{2}$ are not derived equivalent. Indeed, since the quiver of $\Lambda_{2}$ has an oriented cycle, the algebra $\Lambda_{2}$ cannot be derived equivalent to a hereditary algebra.
\end{ex}

\medskip

\subsection{Generalized Kauer moves}

\medskip

In this subsection, we focus on the notion of generalized Kauer moves. These are particular moves of edges in a Brauer graph that yield a derived equivalence between their corresponding Brauer graph algebras \cite[Theorem 3.10]{Soto}. In order to apply the results of the previous section, we need to generalize this derived equivalence to the graded case. Let us first recall their construction from \cite{Soto}. We start with an alternative definition of a Brauer graph arising from combinatorial topology (see \cite{Lazarus} for instance).

\medskip

\begin{defn} \label{def:Brauer graph H}
A \textit{Brauer graph} is the data $\Gamma=(H,\iota,\sigma)$ where 

\smallskip

\begin{itemize}[label=\textbullet, font=\tiny]
\item $H$ is the set of half-edges;
\item $\iota$ is a permutation of $H$ without fixed point satisfying $\iota^{2}=\id{H}$ : it is called the \textit{pairing};
\item $\sigma$ is a permutation of $H$ called the \textit{orientation}.
\end{itemize}

\end{defn}

\medskip

To the data $\Gamma=(H,\iota,\sigma)$, one can naturally construct a graph whose vertex set is $\Gamma_{0}=H/\sigma$, edge set is $\Gamma_{1}=H/\iota$ and source map is the natural projection $s:H\rightarrow H/\sigma$. Moreover, each cycle in the decomposition of $\sigma$ corresponds to a cyclic ordering of the edges around a vertex. We keep previous convention for the orientation of a Brauer graph (see before Example \ref{ex:Brauer graph}).

\medskip

\begin{ex} \label{ex:new Brauer graph}
Let $\Gamma$ be the Brauer graph defined in Example \ref{ex:Brauer graph}. Using Definition \ref{def:Brauer graph H}, this Brauer graph may be seen as $\Gamma=(H,\iota,\sigma)$ 

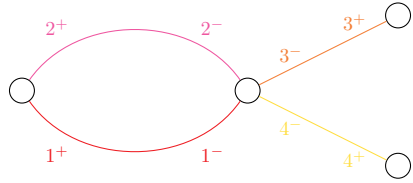
\smallskip

     \begin{figure}[H]
        \centering
        \begin{tikzpicture}[scale=1]
        \tikzstyle{vertex}=[circle,draw]
        \node[vertex] (a) at (-1.5,0) {};
        \node[vertex] (b) at (1.5,0) {};
        \node[vertex] (c) at (3.5,1) {};
        \node[vertex] (d) at (3.5,-1) {};
        \draw[Red] (a) to[bend right=55] node[near start,below left, scale=0.75]{$1^{+}$} node[near end, below right, scale=0.75]{$1^{-}$} (b);
        \draw[VioletRed] (a) to[bend left=55] node[near start, above left, scale=0.75]{$2^{+}$} node[near end, above right, scale=0.75]{$2^{-}$} (b); 
        \draw[Orange] (b)--(c) node[near start, above, scale=0.75]{$3^{-}$} node[near end, above, scale=0.75]{$3^{+}$};
        \draw[Goldenrod] (b)--(d) node[near start, below, scale=0.75]{$4^{-}$} node[near end, below, scale=0.75]{$4^{+}$};
        \end{tikzpicture}
    \end{figure}
    
\smallskip

\noindent where $H=\{1^{+}, 1^{-}, 2^{+}, 2^{-}, 3^{+}, 3^{-}, 4^{+}, 4^{-}\}$, $\iota=(1^{+}\ 1^{-})(2^{+}\ 2^{-})(3^{+}\ 3^{-})(4^{+}\ 4^{-})$ and $\sigma=(1^{+}\ 2^{+})(1^{-}\ 4^{-}\ 3^{-}\ 2^{-})$. Note that each cycle in $\sigma$ corresponds to a cyclic ordering of the half edges around a vertex that was computed in Example \ref{ex:Brauer graph}.
\end{ex}

\medskip

Let us recall now the notion of $G$-graded Brauer graph where $G$ is an abelian group.

\medskip

\begin{defn} \label{def:homogeneous grading on Brauer graph algebras}
    Let $\Gamma=(H,\iota,\sigma)$ be a Brauer graph. 

    \smallskip

    \begin{itemize}[label=\textbullet,font=\tiny]
    \item For $g\in G$, we say that a $G$-grading $d:H\rightarrow G$ is \textit{$g$-homogeneous} if for all $v\in H/\sigma$
    
    \smallskip
    
    \[\sum_{h\in H, s(h)=v}d(h)=g\]

    \smallskip

    \item We say that $(\Gamma,d)$ is a \textit{$G$-graded Brauer graph} if the $G$-grading $d:H\rightarrow G$ is $g$-homogeneous for some $g\in G$. 
    \end{itemize}
\end{defn}

\medskip

\begin{ex}
    Let $\Gamma=(H,\iota,\sigma)$ be a Brauer graph and $B$ be its Brauer graph algebra. One can check that an admissible cut of $B$ can be understood as the data of a $\Z$-grading $d:H\rightarrow \Z$ which is 1-homogeneous and whose values are contained in $\{0,1\}$.
\end{ex}

\medskip

    Let $(\Gamma,d)$ be a $G$-graded Brauer graph and $B$ be the Brauer graph algebra associated to $\Gamma$. Note that  $d$ induces a $G$-grading on the quiver of $B$ by identifying the arrows with their start half-edge. Since $d$ is $g$-homogeneous for some $g\in G$, the relations of $B$ are homogeneous with respect to $d$. Thus, $d$ gives rise to a $G$-grading on $B$. We denote by $(B,d)$ the corresponding $G$-graded algebra.

\medskip

\begin{defn}
    Let $(\Gamma,d)=(H,\iota,\sigma,d)$ be a $G$-graded Brauer graph. We fix $h\in H$ and $r\in\Z_{\ge 0}$ smaller than the size of the $\sigma$-orbit of $h$. The \textit{$G$-graded generalized Kauer move} of $(\Gamma,d)$ over $(h,r)$ is the $G$-graded Brauer graph

\[\mu^{+}_{(h,r)}(\Gamma,d)=(H,\iota,\sigma_{(h,r)},d_{(h,r)})\]

\noindent where $\sigma_{(h,r)}=(h \quad \sigma^{r+1}h)\,\sigma\, (\sigma^{r}h \quad \iota\sigma^{r+1}h)$ and $d_{(h,r)}:H\rightarrow G$ is defined by

\smallskip

\begin{equation*}
    \begin{aligned}
        d_{(h,r)}: \ &\iota\sigma^{r+1}h &&\longmapsto &&-\sum_{i=0}^{r}d(\sigma^{i}h) \\
        &\sigma^{r}h &&\longmapsto &&\left\{\begin{aligned}
            &d(\iota\sigma^{r+1}h)+d(\sigma^{r}h) &&\mbox{if $\iota\sigma^{r+1}h\neq\sigma^{-1}h$} \\
            &\sum_{i=-1}^{r}d(\sigma^{i}h)+d(\sigma^{r}h) &&\mbox{else}
        \end{aligned}\right. \\
        &\sigma^{-1}h &&\longmapsto &&\left\{\begin{aligned}
            &\sum_{i=-1}^{r}d(\sigma^{i}h) &&\mbox{if $\iota\sigma^{r+1}h\neq\sigma^{-1}h$} \\
            &-\sum_{i=0}^{r}d(\sigma^{i}h) &&\mbox{else}   
        \end{aligned}\right. \\
        &h' &&\longmapsto &&d(h') \qquad \mbox{for $h'\neq\iota\sigma^{r+1}h,\sigma^{r}h,\sigma^{-1}h$}
    \end{aligned}
\end{equation*}
\end{defn}

\smallskip

Denoting respectively by $s:H\rightarrow H/\sigma$ and $s_{(h,r)}:H\rightarrow H/\sigma_{(h,r)}$ the source maps of $(H,\iota,\sigma)$ and $(H,\iota,\sigma_{(h,r)})$, one can easily check that there exists a bijection $\phi:H/\sigma\rightarrow H/\sigma_{(h,r)}$ such that for all $v\in H/\sigma_{(h,r)}$

\smallskip

\[\sum_{h\in H, s_{(h,r)}(h)=v}d_{(h,r)}=\sum_{h\in H, s(h)=\phi^{-1}(v)}d(h)\]

\smallskip

\noindent Hence, $\mu^{+}_{(h,r)}(\Gamma,d)$ is a well-defined $G$-graded Brauer graph. Moreover the underlying Brauer graph $\mu^{+}_{(h,r)}(\Gamma):=(H,\iota,\sigma_{(h,r)})$ of $\mu^{+}_{(h,r)}(\Gamma,d)$ is obtained from $\Gamma$ as follows

\smallskip

\begin{figure}[H]
\centering
        \begin{tikzpicture}[scale=0.8]
        \tikzstyle{vertex}=[draw, circle, minimum size=0.2cm]
        \begin{scope}[xshift=-3cm]
        \node[vertex] (1) at (-2,0) {};
        \node[vertex] (2) at (2,0) {};
        \draw (1)--(2) node[near start, above]{$\sigma^{r+1}h$} node[near end, above]{$\iota\sigma^{r+1}h$};
        \draw (1)--(-2.5,-1.75) node[below]{$\sigma^{-1}h$};
        \draw[Orange] (1)--(-1.5,-1.75) node[below, right]{$h$};
        \draw[Orange] (1)--(-1,-1.25) node[below, right]{$\sigma^{j}h$};
        \draw[Orange] (1)--(-0.5,-0.75) node[below, right]{$\sigma^{r}h$};
        \draw (2)--(2.5,-1.75) node[below]{$\sigma\iota\sigma^{r+1}h$};
        \draw (0,-3.5) node{$\Gamma$};
        \draw[|-|, Orange] (-1.5,-2.2)--(0,-2.2) node[below, midway]{$(h,r)$};
        \end{scope} 
        
        \draw[<->] (0,0)--(3,0);

        \begin{scope}[xshift=6cm]
        \node[vertex] (1) at (-2,0) {};
        \node[vertex] (2) at (2,0) {};
        \draw (1)--(2) node[near start, above]{$\sigma^{r+1}h$} node[near end, above]{$\iota\sigma^{r+1}h$};
        \draw (1)--(-2.5,-1.75) node[below]{$\sigma^{-1}h$};
        \draw[Orange] (2)--(1.5,-1.75) node[below left]{$\sigma^{r}h$};
        \draw[Orange] (2)--(1,-1.25) node[below, left]{$\sigma^{j}h$};
        \draw[Orange] (2)--(0.5,-0.75) node[below, left]{$h$};
        \draw (2)--(2.5,-1.75) node[below]{$\sigma\iota\sigma^{r+1}h$};
        \draw (0,-3.5) node{$\mu^{+}_{(h,r)}(\Gamma)$};
        \end{scope}
        \end{tikzpicture}
\caption{Generalized Kauer move of $(h,r)$}
\label{Generalized Kauer move}
\end{figure}
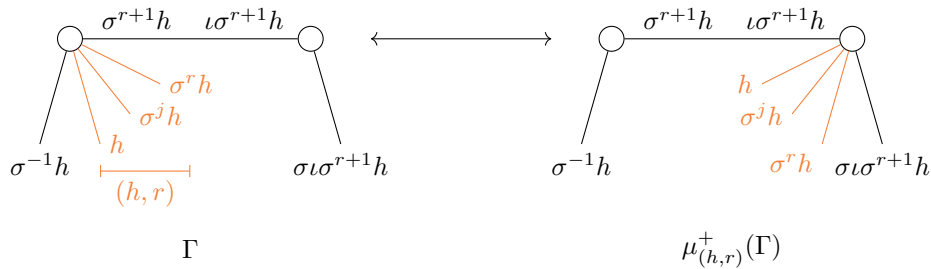

\medskip

Let $(\Gamma,d)=(H,\iota,\sigma,d)$ be a $G$-graded Brauer graph and $H'$ be a subset of $H$ stable under $\iota$. It was proved in \cite{Soto} that one can consider successive $G$-graded generalized Kauer moves over $(h,r)\in H\times\Z_{\ge 0}$ satisfying that $\sigma^{i}h\in H'$ for $i=0,\ldots, r$ and $\sigma^{-1}h,\sigma^{r+1}h\notin H'$. We call such a couple $(h,r)$ a \textit{maximal sector} of elements in $H'$. Moreover, these successive moves can be done in any order. This leads naturally to the following definition.

\begin{defn} \label{def:generalized Kauer move}
    The \textit{$G$-graded generalized Kauer move} of $(\Gamma,d)$ over $H'$ is the $G$-graded Brauer graph $\mu^{+}_{H'}(\Gamma,d)=(H,\iota,\sigma_{H'},d_{H'})$ obtained from the succession of $G$-graded generalized Kauer moves over all maximal sectors $(h,r)\in H\times\Z_{\ge 0}$ of elements in $H'$.
\end{defn}

By construction $\sigma_{H'}$ is given by $\tau_{\mathrm{cut}}^{\Gamma}\sigma\tau_{\mathrm{paste}}^{\Gamma}$ where $\tau_{\mathrm{cut}}^{\Gamma}$ is the product of the transpositions $(h \quad \sigma^{r+1}h)$ and  $\tau_{\mathrm{paste}}^{\Gamma}$ is the product of the transpositions $(\sigma^{r}h \quad \iota\sigma^{r+1}h)$ over all maximal sectors $(h,r)\in H\times \Z_{\ge 0}$ of elements in $H'$. These two permutations are well-defined since the transpositions in the products have pairwise disjoint support.

\medskip

In the non-graded setting, we know that the generalized Kauer moves induce derived equivalences between the corresponding Brauer graph algebras. More precisely, this derived equivalence is described thanks to a tilting object arising from silting mutations. We refer to \cite[Definition 2.30]{AI} for the definition of silting mutation. 

\medskip

\begin{thm}\label{thm:generalized Kauer move}
    Let $\Gamma=(H,\iota,\sigma)$ be a Brauer graph and $H'$ be a subset of $H$ stable under $\iota$. We denote by $\mu^{+}_{H'}(\Gamma)$ the Brauer graph given by $(H,\iota,\sigma_{H'})$ where $\sigma_{H'}$ is defined in Definition \ref{def:generalized Kauer move}. Moreover, we denote by $B$ and $B'$ the Brauer graph algebras associated to $\Gamma$ and $\mu^{+}_{H'}(\Gamma)$ respectively.

    \smallskip

    \begin{enumerate}[label=(\arabic*)]

    \item \label{item:generalized Kauer move 1} \cite{AI} The silting mutation $\mu^{+}(B;e_{H\backslash H'}B)$ is tilting in $\per{B}$ where $e_{H\backslash H'}B$ denotes the direct sum of the indecomposable projective $B$-modules associated to the edges in $(H\backslash H')/\iota$.
    \item \label{item:generalized Kauer move 2} \cite{Soto} The endomorphism algebra of $\mu^{+}(B;e_{H\backslash H'}B)$ is isomorphic to $B'$.
    
    \end{enumerate}

    \smallskip

    \noindent In particular, there is an equivalence of triangulated categories given by

    \smallskip

    \begin{center}
    \begin{tikzcd}[column sep=2.5cm]
        \per{B'} \arrow[cramped]{r}[above, yshift=1mm]{-\underset{B'}{\overset{\mbox{\tiny \bf{L}}}{\otimes}}\ \mu^{+}(B;e_{H\backslash H'}B)} &\per{B}
    \end{tikzcd}
    \end{center}
\end{thm}

\medskip

To apply the results of the previous section, we need a graded version of this theorem. 
    
\medskip

\begin{thm} \label{thm:graded generalized Kauer moves}
    Let $(\Gamma,d)=(H,\iota,\sigma,d)$ be a $G$-graded Brauer graph  and $H'$ be a subset of $H$ stable under $\iota$. We denote by $(B,d)$ and $(B',d')$ the $G$-graded Brauer graph algebras associated to $(\Gamma,d)$ and $\mu^{+}_{H'}(\Gamma,d)$ respectively. 

    \smallskip

    \begin{enumerate}[label=(\arabic*)]
    \item \label{item:graded generalized Kauer moves 1} The silting mutation $\mu^{+}(\oplus_{g\in G}B(g);\oplus_{g\in G}\, e_{H\backslash H'}B(g))$ is tilting in $\per{\gr{B}{\Z}}$ where $e_{H\backslash H'}B$ denotes the direct sum of the indecomposable projective $G$-graded modules over $B$ associated to the edges in $(H\backslash H')/\iota$.

    \item \label{item:graded generalized Kauer moves 2} The silting mutation $\mu^{+}(\oplus_{g\in G}B(g);\oplus_{g\in G}\,e_{H\backslash H'}B(g))$ decomposes as $\oplus_{g\in G}T(g)$ where $T\in\per{\gr{B}{\Z}}$ and whose graded endomorphism algebra is isomorphic to $B'$ as $\Z$-graded algebras.
    
    \end{enumerate}

    \smallskip

    \noindent In particular, there is an equivalence of $\Z$-graded triangulated categories 

    \smallskip

    \begin{center}
        \begin{tikzcd}[column sep=5cm]
            \per{\gr{B'}{\Z}} \arrow[cramped]{r}[above]{-\underset{\mathrm{gr}(B')}{\overset{\mbox{\tiny \bf{L}}}{\otimes}}\ \mu^{+}(\underset{g\in G}{\bigoplus}B(g) \, ; \, \underset{g\in G}{\bigoplus}e_{H\backslash H'}B(g))} &\per{\gr{B}{\Z}}
        \end{tikzcd}
    \end{center}

    \smallskip

    \noindent where $-\otimes_{\mathrm{gr}(B')} \, \mu^{+}(\oplus_{g\in G}\, B(g) \, ; \, \oplus_{g\in G}\, e_{H\backslash H'}B(g)):\mathrm{K^{b}}(\gr{B'}{\Z})\rightarrow\mathrm{K^{b}}(\gr{B}{\Z})$ is the graded tensor product between the bounded homotopy category of $\gr{B'}{\Z}$ and $\gr{B}{\Z}$ as defined in \cite{ZH}.
\end{thm}

\medskip

\begin{proof}
    In the following, $e_{[h]}B$ denotes the indecomposable projective $G$-graded module over $B$ corresponding to the edge $[h]\in H/\iota$. For all $h\in H$, we denote by

\smallskip

\[\alpha(h,H'):e_{[h]}B\longrightarrow e_{[\sigma^{r(h)+1}h]}B(\sum_{i=0}^{r(h)+1}d(\sigma^{i}h))\]

\smallskip

\noindent the morphism in $\per{\gr{B}{\Z}}$ induced by the path from $h$ to $\sigma^{r(h)+1}h$ where $r(h)+1=\min\{r\ge 0 \ | \ \sigma^{r}h\notin H'\}$. By an abuse of notation, we set $e_{[\sigma^{r(h)+1}h]}B$ to be 0 if all elements of the $\sigma$-orbit of $h$ is contained in $H'$. Note that the left minimal approximation of $e_{[h]}B(g)$ in $\per{\gr{B}{\Z}}$ is given by

\smallskip
    
    \begin{center}
        \begin{tikzcd}[column sep=2cm, row sep=0.1cm]
            f_{[h]}(g): e_{[h]}B(g) \arrow{r}{\begin{pmatrix} \alpha(h,H')(g) \\ \alpha(\iota h,H')(g)\end{pmatrix}} &e_{[\sigma^{r(h)+1}h]}B(g+\sum_{i=0}^{r(h)+1}d(\sigma^{i}h))) \ \bigoplus \\
            &e_{[\sigma^{r(\iota h)+1}\iota h]}B(g+\sum_{i=0}^{r(\iota h)+1}d(\sigma^{i}\iota h)))
        \end{tikzcd}
    \end{center}

\smallskip

\noindent for all $[h]\in H/\iota$. Thus, the left mutation of $\oplus_{g\in G}B(g)$ over $\oplus_{g\in G}\, e_{H\backslash H'}B(g)$ can be decomposed into $\oplus_{g\in G}T(g)$ where 

\smallskip

\begin{center}
    \begin{tikzcd}
        T=e_{H\backslash H'}B \oplus \mathrm{Cone}(f_{[h]})\in\per{\gr{B}{\Z}}
    \end{tikzcd}
\end{center}

\smallskip

Thanks to Theorem 2.32 of \cite{AI}, it suffices to show that for all $[h']\in (H\backslash H')/\iota$, the morphism $\Hom[\per{\gr{B}{\Z}}]{e_{[h']}B(g)}{f_{[h]}}$ is injective for all $g\in G$ in order to prove \ref{item:graded generalized Kauer moves 1}. Let us consider $a:e_{[h']}B(g)\rightarrow e_{[h]}B$ in $\per{\gr{B}{\Z}}$ such that the following diagram commutes

\smallskip

\begin{center}
    \begin{tikzcd}[column sep=2cm, row sep=1.7cm]
        e_{[h]}B \arrow[cramped]{r}[above]{f_{[h]}} &e_{[\sigma^{r(h)+1}h]}B \bigoplus e_{[\sigma^{r(\iota h)+1}\iota h]}B \\
        e_{[h']}B(g) \arrow[cramped]{u}[left]{a} \arrow[cramped]{ur}[below, right,yshift=-2mm]{0} &
    \end{tikzcd}
\end{center}

\smallskip

\noindent We want to prove that $a=0$. By assumption, we know that $d:H\rightarrow \Z$ is $g_{0}$-homogeneous, for some $g_{0}\in G$. By Theorem \ref{thm:Schroll}, $B$ is the trivial extension of some algebra $\Lambda$. Moreover, $d$ induces a $\Z$-grading on $\Lambda$ as seen before Corollary \ref{cor:gradué BGA}. One can check that the $g$-th degree component of $(B,d)$ is given by $\Lambda_{g}\,\oplus\,D(\Lambda_{g_{0}-g})$. Hence, there is an isomorphism of $G$-graded $B$-$B$-bimodule $\psi_{g'}:B(g')\rightarrow DB(g'-g_{0})$ for all $g'\in G$. In particular, $\psi_{g}$ induces an isomorphism

\[\psi_{g}:\Hom[\per{\gr{B}{\Z}}]{e_{[h_{1}]}B(g)} {e_{[h_{2}]}B}\overset{\sim}{\longrightarrow} D\Hom[\per{\gr{B}{\Z}}]{e_{[h_{2}]}B}{e_{[h_{1}]}B(g-g_{0})}\]

\noindent for all $[h_{1}],[h_{2}]\in H/\iota$. Moreover, for all $b:e_{[h]}B\rightarrow e_{[h']}B(g-g_{0})$ in $\per{\gr{B}{\Z}}$, there exists $b':e_{[\sigma^{r(h)+1}h]}B\oplus e_{[\sigma^{r(\iota h)+1}\iota h]}B \rightarrow e_{[h']}B(g-g_{0})$ in $\per{\gr{B}{\Z}}$ such that $b=b'f_{[h]}$ since $f_{[h]}$ is a left $\mathrm{add}(\oplus_{g'\in G}\, e_{H\backslash H'}B(g'))$-approximation of $e_{[h]}B$ in $\per{\gr{B}{\Z}}$. Hence, we obtain

\[\psi_{g}(a)(b)=\psi_{g_{0}}(ba)(\id{e_{[h']}B})=\psi_{g_{0}}(b'f_{[h]}a)(\id{e_{[h']}B})=0\]

\noindent We have proved that $\psi_{g}(a)=0$ which induces that $a=0$ since $\psi_{g}$ is injective.

\medskip

Note that the image of $T$ through the forgetful functor $\per{\gr{B}{\Z}}\rightarrow\per{B}$ is equal to $\mu^{+}(B;e_{H\backslash H'}B)$ whose endomorphism ring is isomorphic to $B'$ thanks to the second point of Theorem \ref{thm:generalized Kauer move}. By construction of $d'$, one can check that the graded endomorphism algebra of $T$ is isomorphic to $(B',d')$ as $G$-graded algebras. Hence, we obtain the desired equivalence using Theorem 1 of \cite{ZH}. Moreover, this equivalence commutes with the shift functor thanks to Theorem 2 of \cite{ZH}. 
\end{proof}

\medskip

One can now apply Corollary \ref{cor:non gradué} using $\Z$-graded generalized Kauer moves. 

\medskip

\begin{cor} \label{cor:non graded generalized Kauer moves}
Let $\Gamma_{1}=(H_{1},\iota_{1},\sigma_{1})$ and $\Gamma_{2}=(H_{2},\iota_{2},\sigma_{2})$ be two Brauer graph algebras equipped with an admissible cut $d_{i}:H_{i}\rightarrow \Z$. We denote by $(B_{i},d_{i})$ the $\Z$-graded Brauer graph algebra corresponding to $(\Gamma_{i},d_{i})$. Let us assume that there exists $H_{1}'$ a subset of $H_{1}$ stable under $\iota_{1}$ such that $\mu^{+}_{H_{1}'}(\Gamma_{1},d_{1})=(\Gamma_{2},\delta_{2})$ for some 1-homogeneous $\Z$-grading $\delta_{2}:H_{2}\rightarrow\Z$. Moreover, assume that for all $j$ there exists $n_{j}\in \Z$ such that

\smallskip

\begin{equation*}
    \begin{aligned}
    \proj{B_{2},\delta_{2}}{\Z} &\overset{\sim}{\longrightarrow} \proj{B_{2},d_{2}}{\Z} \\[0.1cm]
    e_{j}B' &\longmapsto e_{j}B'(n_{j})
    \end{aligned}
    \end{equation*}

    \smallskip
    
\noindent gives an equivalence of $\Z$-graded categories. If the gentle algebras $\Lambda_{i}$ corresponding to $(\Gamma_{i},d_{i})$ have finite global dimension, then these two algebras are derived equivalent.
    
\end{cor}

\medskip

\begin{proof}
    Thanks to Theorem \ref{thm:graded generalized Kauer moves}, we obtain an equivalence of triangulated categories between $\per{\gr{B_{1},d_{1}}{\Z}}$ and $\per{\gr{B_{2},\delta_{2}}{\Z}}$. Hence, our assumption leads to a triangle equivalence between the perfect categories of $\gr{B_{i},d_{i}}{\Z}$. We conclude that $\Lambda_{1}$ and $\Lambda_{2}$ are derived equivalent by Corollary \ref{cor:non gradué}.
\end{proof} 

\medskip

\begin{rem}
    With the notations of the previous corollary, let us assume that the degree 
    of the morphism $\alpha(h,H_{1}')$ is 0 for all $h\in H_{1}$. Then $\delta_{2}:H_{2}\rightarrow \Z$ is an admissible cut of $\Gamma_{2}$ by Proposition 2.6 of \cite{Soto}. We denote by $\Lambda_{3}$ the gentle algebra corresponding to $(\Gamma_{2},\delta_{2})$. Thanks to Theorem 3.8 in \cite{Soto}, we obtain the following triangle equivalence
    
    \smallskip
    
    \begin{center}
        \begin{tikzcd}[column sep=3cm]
            \per{\Lambda_{3}} \arrow[cramped]{r}[above]{\sim}[below]{-\underset{\Lambda_{3}}{\overset{\mbox{\tiny \bf{L}}}{\otimes}}\ \mu^{+}(\Lambda_{1} \, ; \, e_{H_{1}\backslash H'_{1}}\Lambda_{1})} &\per{\Lambda_{1}}
        \end{tikzcd}
    \end{center}

    \smallskip

    \noindent Note that this equivalence does not require $\Lambda_{1}$ or $\Lambda_{3}$ to have finite global dimension. However, in the case where they do have finite global dimension, it is not clear whether this equivalence is the one arising from the commutative diagram induced by Happel's theorem (cf Remark \ref{rem:Happel Rickard}). 

\end{rem}

\medskip

\begin{ex} \label{ex:graded generalized Kauer moves}
    Let $\Gamma=(H,\iota,\sigma)$ be the Brauer graph defined in Example \ref{ex:new Brauer graph} equipped with an admissible cut $d:H\rightarrow\Z$ defined by $d(1^{+})=d(2^{-})=d(3^{+})=d(4^{+})=1$ and $d(1^{-})=d(2^{+})=d(4^{-})=d(3^{-})=0$ which can be represented on $\Gamma$ as follows
    
    \smallskip

    
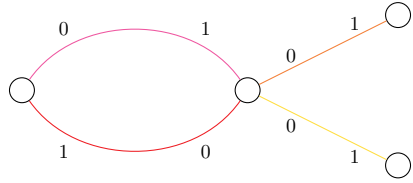
\begin{figure}[H]
        \centering
        \begin{tikzpicture}[scale=1]
        \tikzstyle{vertex}=[circle,draw]
        \node[vertex] (a) at (-1.5,0) {};
        \node[vertex] (b) at (1.5,0) {};
        \node[vertex] (c) at (3.5,1) {};
        \node[vertex] (d) at (3.5,-1) {};
        \draw[Red] (a) to[bend right=55] node[near start,below left, scale=0.75,black]{$1$} node[near end, below right, scale=0.75,black]{$0$} (b);
        \draw[VioletRed] (a) to[bend left=55] node[near start, above left, scale=0.75,black]{$0$} node[near end, above right, scale=0.75,black]{$1$} (b); 
        \draw[Orange] (b)--(c) node[near start, above, scale=0.75,black]{$0$} node[near end, above, scale=0.75,black]{$1$};
        \draw[Goldenrod] (b)--(d) node[near start, below, scale=0.75,black]{$0$} node[near end, below, scale=0.75,black]{$1$};
        \end{tikzpicture}
    \end{figure}

\smallskip
    
    \noindent Considering $H'=\{1^{+},1^{-},4^{-},4^{+}\}$, the $\Z$-graded generalized Kauer move $(\mu^{+}_{H'}(\Gamma),d_{H'})=(H,\iota,\sigma_{H'},d_{H'})$ of $(\Gamma,d)$ over $\Z$ is defined by

    \smallskip

    
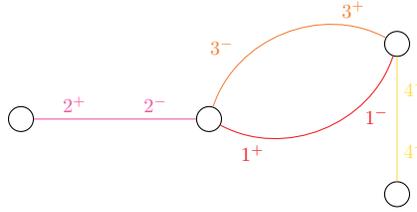
\begin{figure}[H]
        \centering
        \begin{tikzpicture}
        \tikzstyle{vertex}=[circle,draw]
        \node[vertex] (a) at (-1.5,0) {};
        \node[vertex] (b) at (1,0) {};
        \node[vertex] (c) at (3.5,1) {};
        \node[vertex] (d) at (3.5,-1) {};
        \draw[Red] (b) to[bend right=50] node[near start,below left, scale=0.75]{$1^{+}$} node[near end, below right, xshift=-1mm, scale=0.75]{$1^{-}$} (c);
        \draw[Orange] (b) to[bend left=50] node[near start, above left, scale=0.75]{$3^{-}$} node[near end, above right, scale=0.75]{$3^{+}$} (c); 
        \draw[VioletRed] (a)--(b) node[near start, above, scale=0.75]{$2^{+}$} node[near end, above, scale=0.75]{$2^{-}$};
        \draw[Goldenrod] (c)--(d) node[near start, right, scale=0.75]{$4^{-}$} node[near end, right, scale=0.75]{$4^{+}$};
        \end{tikzpicture}
    \end{figure}

    \smallskip

    \noindent where $\sigma_{H'}=(3^{-} \ 2^{-} \ 1^{+})(3^{+} \ 1^{-} \ 4^{-})$ and $d_{H'}(2^{+})=d_{H'}(4^{-})=d_{H'}(4^{+})=1$, $d_{H'}(2^{-})=-1$, $d_{H'}(1^{+})=2$ and $d_{H'}(1^{-})=d_{H'}(3^{-})=d_{H'}(3^{+})=0$. Let us now define an admissible cut $d':H\rightarrow \Z$ on $\mu^{+}_{H'}(\Gamma)$ by $d'(2^{-})=d'(1^{-})=d'(3^{-})=d'(4^{-})=0$ and $d'(2^{+})=d'(1^{+})=d'(3^{+})=d'(4^{+})=1$. These two $\Z$-grading can be represented on $\mu^{+}_{H'}(\Gamma)$ as follows
    
    \smallskip

    \begin{figure}[H]
        \centering
        \begin{tikzpicture}
        \tikzstyle{vertex}=[circle,draw]
        \begin{scope}[xshift=-3.5cm]
         \node[vertex] (a) at (-1.5,0) {};
        \node[vertex] (b) at (1,0) {};
        \node[vertex] (c) at (3.5,1) {};
        \node[vertex] (d) at (3.5,-1) {};
        \draw[Red] (b) to[bend right=50] node[near start,below left, scale=0.75,black]{$2$} node[near end, below right, xshift=-1mm, scale=0.75,black]{$0$} (c);
        \draw[Orange] (b) to[bend left=50] node[near start, above left, scale=0.75,black]{$0$} node[near end, above right, scale=0.75,black]{$0$} (c); 
        \draw[VioletRed] (a)--(b) node[near start, above, scale=0.75,black]{$1$} node[near end, above, scale=0.75,black]{$-1$};
        \draw[Goldenrod] (c)--(d) node[near start, right, scale=0.75,black]{$1$} node[near end, right, scale=0.75,black]{$1$};
        \draw (1,-2) node{$(\mu^{+}_{H'}(\Gamma),d_{H'})$};
        \end{scope}
        \begin{scope}[xshift=3.5cm]
         \node[vertex] (a) at (-1.5,0) {};
        \node[vertex] (b) at (1,0) {};
        \node[vertex] (c) at (3.5,1) {};
        \node[vertex] (d) at (3.5,-1) {};
        \draw[Red] (b) to[bend right=50] node[near start,below left, scale=0.75,black]{$1$} node[near end, below right, xshift=-1mm, scale=0.75,black]{$0$} (c);
        \draw[Orange] (b) to[bend left=50] node[near start, above left, scale=0.75,black]{$0$} node[near end, above right, scale=0.75,black]{$1$} (c); 
        \draw[VioletRed] (a)--(b) node[near start, above, scale=0.75,black]{$1$} node[near end, above, scale=0.75,black]{$0$};
        \draw[Goldenrod] (c)--(d) node[near start, right, scale=0.75,black]{$0$} node[near end, right, scale=0.75,black]{$1$};
        \draw (1,-2) node{$(\mu^{+}_{H'}(\Gamma),d')$};
        \end{scope}
        \end{tikzpicture}
    \end{figure}

    \smallskip
    
    \noindent Denoting by $B'$ the Brauer graph algebra associated to $\mu^{+}_{H'}(\Gamma)$, the two previous $\Z$-gradings are represented on the quiver of $B'$ as follows

    \smallskip
    
    \begin{center}
        \begin{tikzcd}[column sep=2cm, row sep=2cm, scale=0.75]
            1 \arrow[cramped]{r}[above]{0} \arrow[cramped, xshift=-0.7mm, yshift=-0.7mm]{rd}[below left]{2}& 4 \arrow[cramped]{d}[right]{1} \\
            2 \arrow[cramped]{u}[left]{-1} & 3 \arrow[cramped, xshift=0.7mm, yshift=0.7mm]{ul}[above right]{0} \arrow[cramped]{l}[below]{0} \arrow[phantom]{l}[yshift=-7mm, scale=0.9]{(Q_{B'}, d_{H'})}
        \end{tikzcd}
        \hspace{2cm}
         \begin{tikzcd}[column sep=2cm, row sep=2cm, scale=0.75]
            1 \arrow[cramped]{r}[above]{0} \arrow[cramped, xshift=-0.7mm, yshift=-0.7mm]{rd}[below left]{1}& 4 \arrow[cramped]{d}[right]{0} \\
            2 \arrow[cramped]{u}[left]{0} & 3 \arrow[cramped, xshift=0.7mm, yshift=0.7mm]{ul}[above right]{1} \arrow[cramped]{l}[below]{0} \arrow[phantom]{l}[yshift=-7mm, scale=0.9]{(Q_{B'}, d')}
        \end{tikzcd}
        
    \end{center}

    \smallskip
    
    \noindent Hence, we obtain the following equivalence of $\Z$-graded categories

    \smallskip
    
    \begin{equation*}
    \begin{aligned}
    \proj{B',d_{H'}}{\Z} &\overset{\sim}{\longrightarrow} \proj{B',d'}{\Z} \\[0.1cm]
    e_{i}B' &\longmapsto e_{i}B'(-1) &&\mbox{for $i=1,4$} \\
    e_{i}B' &\longmapsto e_{i}B' &&\mbox{for $i=2,3$}
    \end{aligned}
    \end{equation*}

    \smallskip
    
    \noindent Let $\Lambda$ and $\Lambda'$ be the gentle algebras corresponding to $(\Gamma,d)$ and $(\mu^{+}_{H'}(\Gamma),d')$ respectively. By definition, these two algebras are the degree 0 part of $(\Gamma,d)$ and $(\mu^{+}_{H'}(\Gamma),d')$ respectively. Hence, their quivers are given by

    \smallskip

    
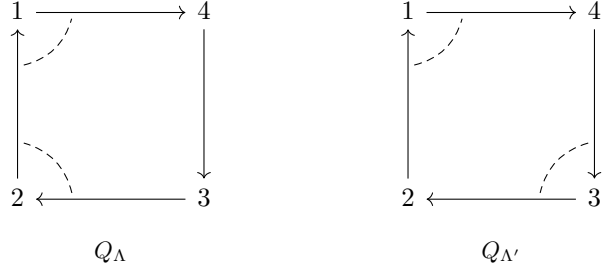
\begin{center}
        \begin{tikzcd}[column sep=2cm, row sep=2cm, scale=0.75]
        1 \arrow[cramped]{r} \arrow[phantom]{r}[near start, name=D]{} & 4 \arrow[cramped]{d}  \\
        2 \arrow[cramped]{u} \arrow[phantom]{u}[near start, name=B]{}[near end, name=C]{}& 3 \arrow[cramped]{l} \arrow[phantom]{l}[near end, name=A]{}
        \arrow[phantom]{l}[yshift=-7mm, scale=0.9]{Q_{\Lambda}} \arrow[dash, densely dashed, cramped, bend right, from=A]{B} \arrow[dash, densely dashed, cramped, bend right, from=C]{D}
        \end{tikzcd}
        \hspace{2cm}
        \begin{tikzcd}[column sep=2cm, row sep=2cm, scale=0.75]
        1 \arrow[cramped]{r} \arrow[phantom]{r}[near start, name=D]{} & 4 \arrow[cramped]{d} \arrow[phantom]{d}[near end, name=A]{} \\
        2 \arrow[cramped]{u} \arrow[phantom]{u}[near end, name=C]{}& 3 \arrow[cramped]{l} \arrow[phantom]{l}[near start, name=B]{}
        \arrow[phantom]{l}[yshift=-7mm, scale=0.9]{Q_{\Lambda'}} \arrow[dash, densely dashed, cramped, bend right, from=A]{B} \arrow[dash, densely dashed, cramped, bend right, from=C]{D}
        \end{tikzcd}
    \end{center}

    \smallskip    
    
    \noindent

    \smallskip

    \noindent Note that the algebras $\Lambda$ and $\Lambda'$ both have finite global dimension. Therefore, we conclude that they are derived equivalent thanks to Corollary \ref{cor:non graded generalized Kauer moves}.
    
\end{ex}

\medskip

Let $\Gamma=(H,\iota,\sigma)$ be a Brauer graph equipped with an admissible cut $d:H\rightarrow \Z$ and a $1$-homogeneous $\Z$-grading $\delta:H\rightarrow \Z$. We denote by $\Lambda$ the gentle algebra corresponding to $(\Gamma,d)$. Hence, $\delta$ induces a $\Z$-grading on the arrows of $\Lambda$ which extends to a $\Z$-grading on $\Lambda$ since its relations are homogeneous with respect to this $\Z$-grading. Thus, by Theorem \ref{thm:Schroll}, it gives rise to a $\Z$-grading $\delta':H\rightarrow \Z$ on the Brauer graph $\Gamma$.   
\medskip

\begin{lem} \label{lem:graded generalized Kauer moves}
    With previous notation, we obtain an equivalence of triangulated categories

    \[\Db{\gr{\mathrm{Triv}(\Lambda),d,\delta}{\Z^{2}}}\simeq \Db{\gr{\mathrm{Triv}(\Lambda),d,\delta'}{\Z^{2}}}\]
\end{lem}

\medskip

\begin{proof}
    This follows from the fact that the following isomorphism

    \smallskip
    
    \begin{equation*}
    \begin{aligned}
        &&\Z^{2} &&&\longrightarrow &&\Z^{2} \\
        &&(x,y) &&&\longmapsto && (x,y-x)
    \end{aligned}
\end{equation*}

    \smallskip

    \noindent sends the $\Z^{2}$-grading $(d,\delta)$ of $\mathrm{Triv}(\Lambda)$ onto the $Z^{2}$-grading $(d,\delta')$ of $\mathrm{Triv}(\Lambda)$. 
\end{proof}

\medskip

We can now apply Corollary \ref{cor:gradué} using $\Z^{2}$-graded generalized Kauer moves.

\medskip

\begin{cor} \label{cor:graded generalized Kauer moves}
    Let $\Gamma_{1}=(H_{1},\iota_{1},\sigma_{1})$ and $\Gamma_{2}=(H_{2},\iota_{2},\sigma_{2})$ be two Brauer graph equipped with an admissible cut $d_{i}:H_{i}\rightarrow\Z$ and $\Lambda_{i}$ be the gentle algebra corresponding to $(\Gamma_{i},d_{i})$. We assume that there exists $H_{1}'$ a subset of $H_{1}$ stable under $\iota_{1}$ such that $\mu^{+}_{H_{1}'}(\Gamma_{1},d_{1})=(\Gamma_{2},\delta_{2})$ for some 1-homogeneous $\Z$-grading $\delta_{2}:H_{2}\rightarrow \Z$. In this case,

    \smallskip
    
    \begin{enumerate}[label=(\arabic*)]
    \item There exists a 1-homogeneous $\Z$-grading $\delta_{1}:H_{1}\rightarrow \Z$ on $\Gamma_{1}$ such that $\mu^{+}_{H_{1}'}(\Gamma_{1},\delta_{1})=(\Gamma_{2},d_{2})$.
    \item If the $\Lambda_{i}$ have finite global dimension then there is an equivalence of triangulated categories

    \smallskip
    
    \[\Db{\gr{\Lambda_{1},\delta_{1}}{\Z}}\simeq \Db{\gr{\Lambda_{2},\delta_{2}}{\Z}}\]
    \end{enumerate}
\end{cor}

\medskip

\begin{proof}
    The first point is clear by definition of the grading in a $\Z$-graded generalized Kauer move. Let $B_{i}$ be the Brauer graph algebra associated to $\Gamma_{i}$. By definition, $(\Gamma_{2},\delta_{2},d_{2})$ is the $\Z^{2}$-graded generalized Kauer move of $(\Gamma_{1},d_{1},\delta_{1})$ over $H_{1}'$. Hence, by Theorem \ref{thm:graded generalized Kauer moves}, there is an equivalence of triangulated categories

    \[\per{\gr{B_{1},d_{1},\delta_{1}}{\Z^{2}}}\simeq \per{\gr{B_{2},d_{2},\delta_{2}}{\Z^{2}}}\]
    
    \noindent As seen before Lemma \ref{lem:graded generalized Kauer moves}, $\delta_{i}$ induces a $\Z$-grading on $\Lambda_{i}$. This $\Z$-grading gives rise to a $\Z$-grading on $B_{i}=\T{\Lambda_{i}}$ which will be denoted by $\delta_{i}'$. Then, the previous lemma leads to another triangle equivalence

    \[\per{\gr{B_{1},d_{1},\delta_{1}'}{\Z^{2}}}\simeq \per{\gr{B_{2},d_{2},\delta_{2}'}{\Z^{2}}}\]

    \noindent Hence, we obtain the second point thanks to Corollary \ref{cor:gradué}. 
\end{proof}

\medskip

\begin{ex} \label{ex:double graded generalized Kauer moves}
    Let us consider $(\Gamma,d)$ and $(\mu^{+}_{H'}(\Gamma),d_{H'})$ defined in Examples \ref{ex:graded generalized Kauer moves}. We consider the admissible cut $d':H\rightarrow\Z$ on $\mu^{+}_{H'}(\Gamma)$ defined by $d'(2^{+})=d'(1^{+})=d'(1^{-})=d(4^{+})=1$ and $d'(4^{-})=d'(3^{-})=d'(2^{-})=d'(3^{+})=0$. Thus, one can check that the 1-homogeneous $\Z$-grading $\delta:H\rightarrow\Z$ satisfying $\mu^{+}_{H'}(\Gamma,\delta)=(\mu^{+}_{H'}(\Gamma),d')$ is defined by $\delta(4^{-})=-1$, $\delta(2^{-})=\delta(3^{+})=\delta(1^{-})=\delta(2^{+})=\delta(4^{+})=1$ and $\delta(3^{-})=\delta(1^{+})=0$. Hence, the $\Z^{2}$-graded graphs $(\Gamma,d,\delta)$ and $(\mu^{+}_{H'}(\Gamma),d',d_{H'})$ are given by

    \smallskip

    
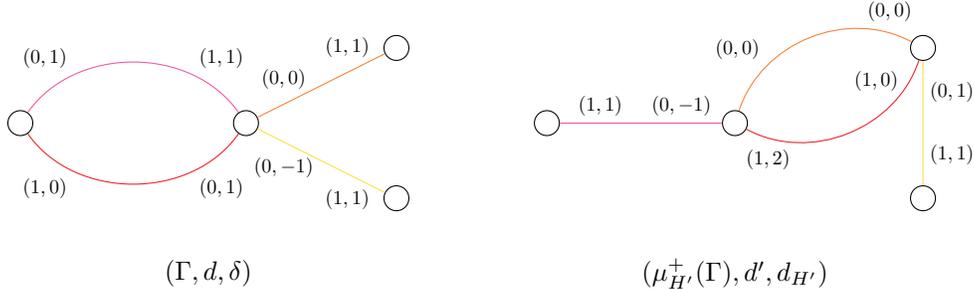
\begin{figure}[H]
        \centering
        \begin{tikzpicture}
        \tikzstyle{vertex}=[circle,draw]
        \begin{scope}[xshift=-3.5cm]
        \node[vertex] (a) at (-1.5,0) {};
        \node[vertex] (b) at (1.5,0) {};
        \node[vertex] (c) at (3.5,1) {};
        \node[vertex] (d) at (3.5,-1) {};
        \draw[Red] (a) to[bend right=55] node[near start,below left, scale=0.75,black]{$(1,0)$} node[near end, below right, scale=0.75,black]{$(0,1)$} (b);
        \draw[VioletRed] (a) to[bend left=55] node[near start, above left, scale=0.75,black]{$(0,1)$} node[near end, above right, scale=0.75,black]{$(1,1)$} (b); 
        \draw[Orange] (b)--(c) node[near start, above, scale=0.75,black,yshift=1mm,xshift=-1mm]{$(0,0)$} node[near end, above, scale=0.75,black,yshift=1mm,xshift=-1mm]{$(1,1)$};
        \draw[Goldenrod] (b)--(d) node[near start, below, scale=0.75,black,yshift=-1mm,xshift=-1mm]{$(0,-1)$} node[near end, below, scale=0.75,black,yshift=-1mm,xshift=-1mm]{$(1,1)$};
        \draw (1,-2) node{$(\Gamma,d,\delta)$};
        \end{scope}
        \begin{scope}[xshift=3.5cm]
         \node[vertex] (a) at (-1.5,0) {};
        \node[vertex] (b) at (1,0) {};
        \node[vertex] (c) at (3.5,1) {};
        \node[vertex] (d) at (3.5,-1) {};
        \draw[Red] (b) to[bend right=50] node[near start,below left, scale=0.75,black]{$(1,2)$} node[near end, above, xshift=-2mm, yshift=1mm, scale=0.75,black]{$(1,0)$} (c);
        \draw[Orange] (b) to[bend left=50] node[near start, above left, scale=0.75,black]{$(0,0)$} node[near end, above right, scale=0.75,black]{$(0,0)$} (c); 
        \draw[VioletRed] (a)--(b) node[near start, above, scale=0.75,black]{$(1,1)$} node[near end, above, scale=0.75,black]{$(0,-1)$};
        \draw[Goldenrod] (c)--(d) node[near start, right, scale=0.75,black]{$(0,1)$} node[near end, right, scale=0.75,black]{$(1,1)$};
        \draw (1,-2) node{$(\mu^{+}_{H'}(\Gamma),d',d_{H'})$};
        \end{scope}
        \end{tikzpicture}
    \end{figure}
    
    \smallskip

    \noindent In particular, the quivers corresponding to $\Gamma$ and $\mu^{+}_{H'}(\Gamma)$ have the following $\Z^{2}$-grading

    \smallskip

    
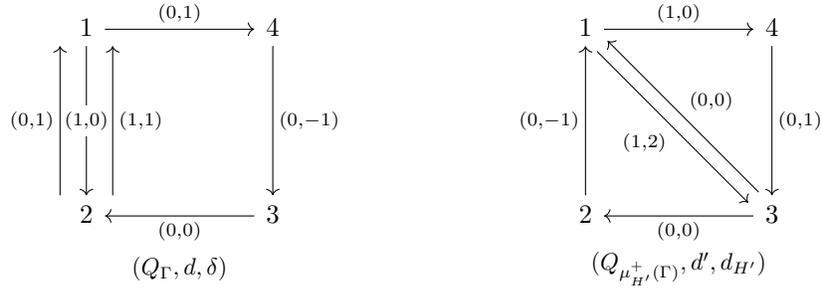
\begin{center}
         \begin{tikzcd}[column sep=2cm, row sep=2cm, scale=0.75]
            1 \arrow[cramped]{r}[above]{(0,1)} \arrow[cramped]{d}[description]{(1,0)} & 4 \arrow[cramped]{d}[right]{(0,-1)} \\
            2 \arrow[cramped, xshift=3.5mm]{u}[right]{(1,1)} \arrow[cramped, xshift=-3.5mm]{u}[left]{(0,1)}   & 3 \arrow[cramped]{l}[below]{(0,0)} \arrow[phantom]{l}[yshift=-7mm, scale=0.9]{(Q_{\Gamma}, d,\delta)}
        \end{tikzcd}
        \hspace{2cm}
         \begin{tikzcd}[column sep=2cm, row sep=2cm, scale=0.75]
            1 \arrow[cramped]{r}[above]{(1,0)} \arrow[cramped, xshift=-0.7mm, yshift=-0.7mm]{rd}[below left]{(1,2)}& 4 \arrow[cramped]{d}[right]{(0,1)} \\
            2 \arrow[cramped]{u}[left]{(0,-1)} & 3 \arrow[cramped, xshift=0.7mm, yshift=0.7mm]{ul}[above right]{(0,0)} \arrow[cramped]{l}[below]{(0,0)} \arrow[phantom]{l}[yshift=-7mm, scale=0.9]{(Q_{\mu^{+}_{H'}(\Gamma)}, d', d_{H'})}
        \end{tikzcd}
        
    \end{center}
    
    \smallskip    
    
    \noindent Let $\Lambda$ and $\Lambda'$ be the gentle algebras corresponding to $(\Gamma,d)$ and $(\mu^{+}_{H'}(\Gamma),d')$ respectively. By definition, they are respectively the degree zero part of the $\Z$-graded Brauer graph algebras associated to $(\Gamma,d)$ and $(\mu^{+}_{H'}(\Gamma),d')$. As seen before Lemma \ref{lem:graded generalized Kauer moves}, $\delta$ and $d_{H'}$ induce a $\Z$-grading on $\Lambda$ and $\Lambda'$ respectively. These $\Z$-gradings are given on the quivers by

    \smallskip

    
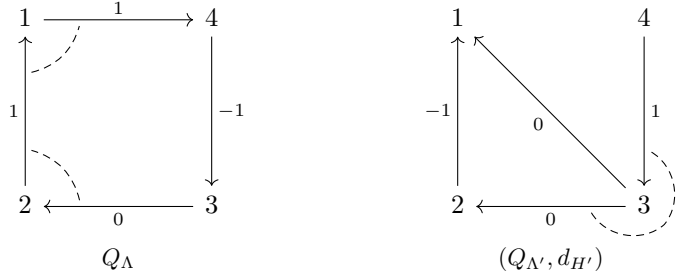
\begin{figure}[H]
    \centering
    \begin{tikzcd}[column sep=2cm, row sep=2cm, scale=0.75]
        1 \arrow[cramped]{r}[above]{1} \arrow[phantom]{r}[near start, name=D]{} & 4 \arrow[cramped]{d}[right]{-1}  \\
        2 \arrow[cramped]{u}[left]{1} \arrow[phantom]{u}[near start, name=B]{}[near end, name=C]{}& 3 \arrow[cramped]{l}[below]{0}\arrow[phantom]{l}[near end, name=A]{}
        \arrow[phantom]{l}[yshift=-7mm, scale=0.9]{Q_{\Lambda}} \arrow[dash, densely dashed, cramped, bend right, from=A]{B} \arrow[dash, densely dashed, cramped, bend right, from=C]{D}
        \end{tikzcd}
    \hspace{2cm}
    \begin{tikzcd}[column sep=2cm, row sep=2cm, scale=0.75]
        1 &4 \arrow[cramped]{d}[right]{1} \\
        2 \arrow[cramped]{u}[left]{-1} \arrow[phantom]{r}[near end, name=A]{} &3 \arrow[cramped]{l}[below]{0} \arrow[phantom]{u}[near start, name=B]{} \arrow[cramped]{ul}[below left]{0} \arrow[phantom]{l}[yshift=-7mm, scale=0.9]{(Q_{\Lambda'},d_{H'})} \arrow[dash, densely dashed, cramped, to path={(A) to[bend right=50] (2.01,-2.01) to[bend right=50] (B)}]
    \end{tikzcd}
    \end{figure}
    
    \noindent Note that the algebras $\Lambda$ and $\Lambda'$ both have finite global dimension. Hence, there is a triangle equivalence thanks to Corollary \ref{cor:graded generalized Kauer moves}

    \[\Db{\gr{\Lambda,\delta}{\Z}}\simeq \Db{\gr{\Lambda',d_{H'}}{\Z}}\]
    
    \noindent However, note that the algebras $\Lambda$ and $\Lambda'$ are not derived equivalent. Indeed, one can check that the AG-invariant defined in \cite{AG} of the algebras $\Lambda$ and $\Lambda'$ are different.
\end{ex}

\medskip

\subsection{Triangular matrix algebra}

\medskip

In this part, we construct derived equivalent triangular matrix gentle algebras from their corresponding Brauer graph algebras. One can apply the results of the previous section to obtain the tilting object associated to this derived equivalence in the non-graded and in the graded case.  

\medskip

Let $\Gamma$ be a Brauer graph and $B$ be its corresponding Brauer graph algebra. Assume that there exist two idempotents $e^{+}$ and $e^{-}$ of $B$ such that

\smallskip

\begin{enumerate}[label=\textbullet, font=\tiny]
    \item $1_{B}=e^{+}+e^{-}$;
    \item there exist vertices $v_{1},\ldots,v_{n}$ in $\Gamma$ and arrows $\alpha_{i}:e^{-}\rightarrow e^{+}$ and $\beta_{i}:e^{+}\rightarrow e^{-}$ arising from the cyclic ordering at $v_{i}$ such that $B/<\alpha_{i},\beta_{i}>$ has two connected component, where $<\alpha_{i},\beta_{i}>$ is the ideal of $B$ generated by the $\alpha_{i}$ and $\beta_{i}$.
\end{enumerate}

\smallskip

\noindent In other words, the algebra $B$ admits the following decomposition

\smallskip

\begin{equation*}
    B=\begin{pmatrix}
        e^{+}Be^{+} & e^{+}Be^{-} \\[0.9em]
        e^{-}Be^{+} & e^{-}Be^{-}
    \end{pmatrix}
\end{equation*}

\smallskip

\noindent where $e^{+}Be^{-}$ is the $(e^{+}Be^{+})$\,-\,$(e^{-}Be^{-})$-bimodule generated by $\alpha_{1},\ldots\alpha_{n}$ and $e^{-}Be^{+}$ is the $(e^{-}Be^{-})$\,-\,$(e^{+}Be^{+})$-bimodule generated by $\beta_{1},\ldots,\beta_{n}$. Hence, $\Gamma$ is locally given by

\smallskip

    
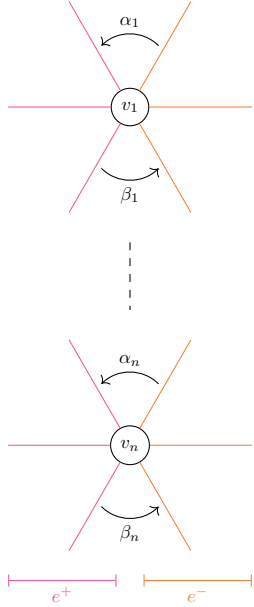
\begin{center}
    \begin{tikzpicture}[scale=0.9]
    \begin{scope}[yshift=3cm]
    \node[draw,circle, scale=0.7] (A) at (0,0) {$v_{1}$};
    \draw[Orange] (A) to (0:1.8); 
    \draw[Orange] (A)--(60:1.8);
    \draw[VioletRed] (A)--(120:1.8);
    \draw[VioletRed] (A)--(180:1.8);
    \draw[VioletRed] (A)--(240:1.8);
    \draw[Orange] (A)--(300:1.8);
    \draw[->] (65:1) to[bend right=45] node[above, scale=0.7]{$\alpha_{1}$} (115:1);
    \draw[->] (245:1) to[bend right=45] node[below, scale=0.7]{$\beta_{1}$} (295:1);
    \end{scope}
    \draw[dashed] (0,1)--(0,0);

    \begin{scope}[yshift=-2cm]
    \node[draw,circle, scale=0.7] (A) at (0,0) {$v_{n}$};
    \draw[Orange] (A) to (0:1.8); 
    \draw[Orange] (A)--(60:1.8);
    \draw[VioletRed] (A)--(120:1.8);
    \draw[VioletRed] (A)--(180:1.8);
    \draw[VioletRed] (A)--(240:1.8);
    \draw[Orange] (A)--(300:1.8);
    \draw[->] (65:1) to[bend right=45] node[above, scale=0.7]{$\alpha_{n}$} (115:1);
    \draw[->] (245:1) to[bend right=45] node[below, scale=0.7]{$\beta_{n}$} (295:1);
    \draw[|-|, VioletRed] (-1.8,-2)--(-0.2,-2) node[below, midway, scale=0.7]{$e^{+}$}; 
    \draw[|-|, Orange] (0.2,-2)--(1.8,-2) node[below, midway, scale=0.7]{$e^{-}$};
    \end{scope}
    \end{tikzpicture}
    \end{center}

\smallskip

\noindent Let us consider two admissible cuts $\Delta_{1}$ and $\Delta_{2}$ of the form $\Delta_{1}=\Delta\cup\{\alpha_{1}\ldots,\alpha_{n}\}$ and $\Delta_{2}=\Delta\cup\{\beta_{1},\ldots,\beta_{n}\}$ i.e. these admissible cuts only differ on the vertices $v_{1},\ldots, v_{n}$. In the following, $\Lambda_{i}:=B_{\Delta_{i}}$ denotes the gentle algebras corresponding to these admissible cuts.

\medskip

\begin{prop} \label{prop:tilting object for triangulated gentle algebras}
    Let $B$ be the Brauer graph algebra and $\Lambda_{i}$ be the gentle algebras defined as above. If $\Lambda_{1}$ has finite global dimension, then $T=e^{+}\Lambda_{2}e^{+}\,\oplus\, D(e^{-}\Lambda_{2}e^{-})[1]$ is a tilting object in $\per{\Lambda_{2}}$ whose endomorphism algebra is isomorphic to $\Lambda_{1}$. In particular, we have an equivalence of triangulated categories 
    
 \smallskip

    \begin{center}
        \begin{tikzcd}[column sep=1.5cm]
            \Db{\Lambda_{1}} \arrow[cramped]{r}[above]{-\underset{\Lambda_{1}}{\overset{\mbox{\tiny \bf{L}}}{\otimes}}T} &\Db{\Lambda_{2}}
        \end{tikzcd}
    \end{center}
\end{prop}

\medskip

\begin{proof}
By definition of the admissible cuts of $B$, the gentle algebras $\Lambda_{1}$ and $\Lambda_{2}$ have a decomposition of the form

\smallskip

\begin{equation*}
    \Lambda_{1}=\begin{pmatrix}
        e^{+}\Lambda_{1}e^{+} & 0 \\[0.9em]
        e^{-}\Lambda_{1}e^{+} & e^{-}\Lambda_{1}e^{-}
    \end{pmatrix} \qquad \mbox{and} \qquad
    \Lambda_{2}=\begin{pmatrix}
        e^{+}\Lambda_{2}e^{+} & e^{+}\Lambda_{2}e^{-} \\[0.9em]
        0 & e^{-}\Lambda_{2}e^{-}
    \end{pmatrix}
\end{equation*}

\smallskip

\noindent Note that $e^{+}\Lambda_{1}e^{+}=e^{+}\Lambda_{2}e^{+}$ and $e^{-}\Lambda_{1}e^{-}=e^{-}\Lambda_{2}e^{-}$ by definition of the admissible cuts $\Delta_{i}$. Moreover, one can check that $e^{+}\Lambda_{2}e^{-}=D(e^{-}\Lambda_{1}e^{+})$ since $\alpha_{i}\in D(e^{-}\Lambda_{1}e^{+})$ and $\beta_{i}\in D(e^{+}\Lambda_{2}e^{-})$. Hence, we conclude using Proposition \ref{prop:tilting object for triangulated matrices}. 
\end{proof}

\medskip

\begin{rem} \label{rem:geometric model}
    Using the geometric model of the bounded derived category of a gentle algebra introduced in \cite{OPS}, one can reformulate the previous proposition as follows. Let $\Lambda_{1}$ and $\Lambda_{2}$ be two gentle algebras with the same ribbon graph $\Gamma$. The ribbon surface $S$ associated to $\Gamma$ locally looks like

    \smallskip

    
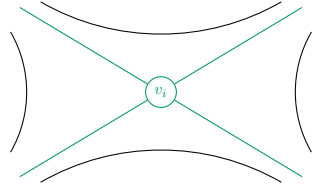
\begin{center}
    \begin{tikzpicture}[scale=0.8]
    \draw (-2,2) arc (-120:-60:4);
    \draw (-2.5,1.5) arc (30:-30:2);
    \draw (2.5,1.5) arc (150:210:2);
    \draw (2,-1) arc (60:120:4);
    \node[draw, circle, ForestGreen, scale=0.6] (A) at (0,0.5) {$v_{i}$};
    \node (B) at (-2.5,2) {};
    \node (C) at (-2.5,-1) {};
    \node (D) at (2.5,2) {};
    \node (E) at (2.5,-1) {};
    \draw[color=ForestGreen] (A) to (B);
    \draw[color=ForestGreen] (A) to (C);
    \draw[color=ForestGreen] (A) to (D);
    \draw[color=ForestGreen] (A) to (E);
    \end{tikzpicture}
    \end{center}

    \smallskip

    \noindent Assume that there exist $v_{1},\ldots,v_{n}$ vertices of $\Gamma$ such that 

    \begin{enumerate}[label=\textbullet, font=\tiny]
        \item  for any vertex $v$ in $\Gamma$ different from the $v_{i}$, the embedding of $\Gamma$ into the marked ribbon surface of $\Lambda_{i}$ is given locally around $v$ by

    \smallskip
    
    \begin{center}
    \begin{tikzpicture}[scale=0.75]
    \begin{scope}[xshift=-4cm]
    \draw (-2,2) arc (-120:-60:4) node[pos=0.5] (A) {};
    \node[draw, circle, ForestGreen, scale=0.75, fill=white] (A') at (A) {$v$};
    \draw (-2.5,1.5) arc (30:-30:2);
    \draw (2.5,1.5) arc (150:210:2);
    \draw (2,-1) arc (60:120:4);
    \node (B) at (-2.5,2) {};
    \node (C) at (-2.5,-1) {};
    \node (D) at (2.5,2) {};
    \node (E) at (2.5,-1) {};
    \draw[color=ForestGreen] (A') to[bend left] (B);
    \draw[color=ForestGreen] (A') to (C);
    \draw[color=ForestGreen] (A') to[bend right] (D);
    \draw[color=ForestGreen] (A') to (E);
    \draw[|-|] (-2.5,-1.5)--++(2.4,0) node[below, midway]{$e^{+}$};
    \draw[|-|] (2.5,-1.5)--++(-2.4,0) node[below, midway]{$e^{-}$};
    \draw (A')+(0,-4.5) node{$S_{\Lambda_{1}}$};
    \end{scope}
    \begin{scope}[xshift=4cm]
    \draw (-2,2) arc (-120:-60:4) node[pos=0.5] (A) {};
    \node[draw, circle, ForestGreen, scale=0.75, fill=white] (A') at (A) {$v$};
    \draw (-2.5,1.5) arc (30:-30:2);
    \draw (2.5,1.5) arc (150:210:2);
    \draw (2,-1) arc (60:120:4);
    \node (B) at (-2.5,2) {};
    \node (C) at (-2.5,-1) {};
    \node (D) at (2.5,2) {};
    \node (E) at (2.5,-1) {};
    \draw[color=ForestGreen] (A') to[bend left] (B);
    \draw[color=ForestGreen] (A') to (C);
    \draw[color=ForestGreen] (A') to[bend right] (D);
    \draw[color=ForestGreen] (A') to (E);
    \draw[|-|] (-2.5,-1.5)--++(2.4,0) node[below, midway]{$e^{+}$};
    \draw[|-|] (2.5,-1.5)--++(-2.4,0) node[below, midway]{$e^{-}$};
    \draw (A')+(0,-4.5) node{$S_{\Lambda_{2}}$};
    \end{scope}
    \end{tikzpicture}
    \end{center}

    \smallskip

    \item the embedding of $\Gamma$ into the marked ribbon surface of $\Lambda_{i}$ is given locally around $v_{i}$ by

    \smallskip
    
    \begin{center}
    \begin{tikzpicture}[scale=0.75]
    \begin{scope}[xshift=-4cm]
    \draw (-2,2) arc (-120:-60:4) node[pos=0.5] (A) {};
    \node[draw, circle, ForestGreen, scale=0.65, fill=white] (A') at (A) {$v_{i}$};
    \draw (-2.5,1.5) arc (30:-30:2);
    \draw (2.5,1.5) arc (150:210:2);
    \draw (2,-1) arc (60:120:4) ;
    \node (B) at (-2.5,2) {};
    \node (C) at (-2.5,-1) {};
    \node (D) at (2.5,2) {};
    \node (E) at (2.5,-1) {};
    \draw[color=ForestGreen] (A') to[bend left] (B);
    \draw[color=ForestGreen] (A') to (C);
    \draw[color=ForestGreen] (A') to[bend right] (D);
    \draw[color=ForestGreen] (A') to (E);
    \draw[|-|] (-2.5,-1.5)--++(2.4,0) node[below, midway]{$e^{+}$};
    \draw[|-|] (2.5,-1.5)--++(-2.4,0) node[below, midway]{$e^{-}$};
    \draw (A')+(0,-4.5) node{$S_{\Lambda_{1}}$};
    \end{scope}
    \begin{scope}[xshift=4cm]
    \draw (-2,2) arc (-120:-60:4);
    \draw (-2.5,1.5) arc (30:-30:2);
    \draw (2.5,1.5) arc (150:210:2);
    \draw (2,-1) arc (60:120:4)  node[pos=0.5] (A) {};
    \node[draw, circle, ForestGreen, scale=0.65, fill=white] (A') at (A) {$v_{i}$};
    \node (B) at (-2.5,2) {};
    \node (C) at (-2.5,-1) {};
    \node (D) at (2.5,2) {};
    \node (E) at (2.5,-1) {};
    \draw[color=ForestGreen] (A') to (B);
    \draw[color=ForestGreen] (A') to[bend right] (C);
    \draw[color=ForestGreen] (A') to (D);
    \draw[color=ForestGreen] (A') to[bend left] (E);
    \draw[|-|] (-2.5,-1.5)--++(2.4,0) node[below, midway]{$e^{+}$};
    \draw[|-|] (2.5,-1.5)--++(-2.4,0) node[below, midway]{$e^{-}$};
    \draw (A')+(0,-2.5) node{$S_{\Lambda_{2}}$};
    \end{scope}
    \end{tikzpicture}
    \end{center}

    \end{enumerate}

    \smallskip

    \noindent For all $v_{i}$, we define a curve $\delta_{i}$ that connects the boundary component containing $v_{i}$ in the marked ribbon surface of $\Lambda_{1}$ to the one containing $v_{i}$ in the marked ribbon surface of $\Lambda_{2}$.

    \smallskip

    
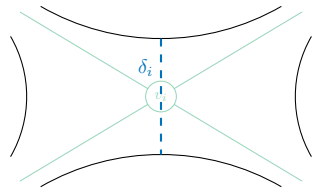
\begin{center}
        \begin{tikzpicture}[scale=0.8]
    \draw (-2,2) arc (-120:-60:4) node[pos=0.5] (F) {};
    \draw (-2.5,1.5) arc (30:-30:2);
    \draw (2.5,1.5) arc (150:210:2);
    \draw (2,-1) arc (60:120:4) node[pos=0.5] (G) {};
    \node[draw, circle, ForestGreen, scale=0.6, opacity=0.4] (A) at (0,0.5) {$v_{i}$};
    \node (B) at (-2.5,2) {};
    \node (C) at (-2.5,-1) {};
    \node (D) at (2.5,2) {};
    \node (E) at (2.5,-1) {};
    \draw[color=ForestGreen, opacity=0.4] (A) to (B);
    \draw[color=ForestGreen, opacity=0.4] (A) to (C);
    \draw[color=ForestGreen, opacity=0.4] (A) to (D);
    \draw[color=ForestGreen, opacity=0.4] (A) to (E);
    \draw[RoyalBlue, thick, dashed] (F.center)--(G.center) node[near start, left, scale=0.75]{$\delta_{i}$}; 
    \end{tikzpicture}
    \end{center}

    \smallskip

    \noindent Assume that the union of the $\delta_{i}$ cuts the ribbon surface $S$ into two subsurfaces. If the marked ribbon surface of $\Lambda_{1}$ has no puncture, then the gentle algebras $\Lambda_{1}$ and $\Lambda_{2}$ are derived equivalent.
\end{rem}

\medskip

Considering the indices modulo 2, we have seen that $\Delta_{i+1}$ induces a $\Z$-grading on $\Lambda_{i}$ which extends onto a $\Z$-grading on $\mathrm{Triv}(\Lambda_{i})$. Hence, one can obtain a $\Z^{2}$-grading on $\mathrm{Triv}(\Lambda_{i})$ whose first component is the natural $\Z$-grading of the trivial extension. 

\medskip

\begin{prop} \label{prop:tilting object for graded triangulated gentle algebra}
    Let $B$ be the Brauer graph algebras and $\Lambda_{i}$ be the gentle algebras defined as before Proposition \ref{prop:tilting object for triangulated gentle algebras}. Assume that $\Lambda_{1}$ has finite global dimension. We denote by $\tau_{i}$ the Auslander-Reiten translation in $\Db{\gr{\Lambda_{i}}{\Z}}$. Then, 

\smallskip

\[{\mathscr{T}}=\bigoplus_{n\in\Z} \ (\tau_{2})^{n}T[2n](-n) \qquad \mbox{with} \qquad T=e^{+}\Lambda_{2}e^{+}\oplus \tau_{2}^{-1}D(e^{-}\Lambda_{2}e^{-})[-1]\]

\smallskip

\noindent is a tilting object in $\per{\gr{\Lambda_{2}}{\Z}}$ satisfying that for all $n\in\Z$, there is an isomorphism of $\Z$-graded $\Lambda_{1}$-modules

\[\bigoplus_{m\in\Z}\Hom[\per{\gr{\Lambda_{2}}{\Z}}]{(\tau_{2})^{m}T[2m](-m)}{(\tau_{2})^{n}T[2n](-n)}\simeq \Lambda_{1}(n)\]

\noindent In particular, we have an equivalence of triangulated categories 

\smallskip

\begin{center}
    \begin{tikzcd}[column sep=1.5cm]
        \Db{\gr{\Lambda_{1}}{\Z}} \arrow[cramped]{r}[above]{-\underset{\mathrm{gr}(\Lambda_{1})}{\overset{\mbox{\tiny \bf{L}}}{\otimes}}\mathscr{T}} &\Db{\gr{\Lambda_{2}}{\Z}}
    \end{tikzcd}
\end{center}

\smallskip

\noindent where $-\otimes_{\mathrm{gr}(\Lambda_{1})}{\mathscr{T}}:\mathrm{K^{b}}(\gr{\Lambda_{1}}{\Z})\rightarrow \mathrm{K^{b}}(\gr{\Lambda_{2}}{\Z})$ is the graded tensor product between the bounded homotopy categories of $\gr{\Lambda_{i}}{\Z}$ as defined in \cite{ZH}.

\end{prop}

\medskip

\phantomsection

\setlength{\bibitemsep}{0pt}
\setlength{\biblabelsep}{8pt}
\defbibnote{nom}{\\ \small{\scshape{Université Grenoble Alpes, CNRS, Institut Fourier, 38610 Gières}} \\ \textit{E-mail address :} \href{mailto:valentine.soto@univ-grenoble-alpes.fr}{valentine.soto@univ-grenoble-alpes.fr}}
\printbibliography[postnote=nom, heading=bibintoc,title={References}]

\end{document}